\documentclass[12pt]{amsart}
\usepackage{amsmath}
\usepackage{amsfonts}
\usepackage{amssymb}
\usepackage[all]{xy}
\usepackage{bbding}
\usepackage{txfonts}
\usepackage{amscd}
\usepackage{xspace}
\usepackage[shortlabels]{enumitem}
\usepackage{ifpdf}
\ifpdf
  \usepackage[colorlinks,final,backref=page,hyperindex]{hyperref}
\else
  \usepackage[colorlinks,final,backref=page,hyperindex,hypertex]{hyperref}
\fi
\usepackage{tikz}
\usepackage[active]{srcltx}

\newtheorem{thm}{Theorem}[section]
\newtheorem{lem}[thm]{Lemma}
\newtheorem{cor}[thm]{Corollary}
\newtheorem{pro}[thm]{Proposition}
\newtheorem{ex}[thm]{Example}
\newtheorem{rmk}[thm]{Remark}
\newtheorem{defi}[thm]{Definition}

\newcommand {\emptycomment}[1]{}
\newcommand{\lon }{\,\rightarrow\,}
\newcommand{\be }{\begin{equation}}
\newcommand{\ee }{\end{equation}}

\newcommand{\Real}{\mathbb R}
\newcommand{\Comp}{\mathbb C}
\newcommand{\Nat}{\mathbb N}
\newcommand{\Integ}{\mathbb Z}
\newcommand{\Field}{\mathbb F}

\newcommand{\huaO}{\mathcal{O}}

\newcommand{\g}{\mathfrak g}

\newcommand{\Id}{\rm{Id}}

\newcommand{\br}[1]{   [ \cdot,    \cdot  ]   }

\newcommand{\dM}{\mathrm{d}}

\newcommand{\ad}{\mathrm{ad}}

\newcommand{\LP}{$\mathsf{LCMod}$~}
\newcommand{\GRB}{\huaO}
\newcommand{\GRBN}{\mathrm{\huaO N}}

\begin{document}

\title[Conformal $r$-matrix-Nijenhuis structures, symplectic-Nijenhuis structures and $\GRBN$-structures]{ Conformal $r$-matrix-Nijenhuis structures, symplectic-Nijenhuis structures and $\GRBN$-structures}

\author{Jiefeng Liu}
\address{School of Mathematics and Statistics, Northeast Normal University,
 Changchun 130024, Jilin, China }
\email{liujf534@nenu.edu.cn}

\author{Sihan Zhou}
\address{School of Mathematics and Statistics, Northeast Normal University,
 Changchun 130024, Jilin, China }
\email{zhoush501@nenu.edu.cn}

\author{Lamei Yuan$^{\ast}$}
\address{School of Mathematics, Harbin Institute of Technology, Harbin 150001, China }
\email{lmyuan@hit.edu.cn}
\thanks{$^{\ast}$ the corresponding author}
\vspace{-5mm}

\begin{abstract}
In this paper, we first study  infinitesimal deformations of a Lie conformal algebra and  a Lie conformal algebra with a module (called an \LP pair), which lead to the notions of Nijenhuis operator on the Lie conformal algebra and Nijenhuis structure on the \LP pair, respectively. Then by adding compatibility conditions between Nijenhuis structures and $\GRB$-operators, we introduce the notion of an $\GRBN$-structure on an \LP pair and show that an $\GRBN$-structure gives rise to a hierarchy of pairwise compatible $\GRB$-operators. In particular, we show that compatible $\huaO$-operators on a Lie conformal algebra can be characterized by Nijenhuis
operators on Lie conformal algebras. Finally, we introduce the notions of conformal $r$-matrix-Nijenhuis structure and symplectic-Nijenhuis structure on the Lie conformal algebra and study their relations.
\end{abstract}





\maketitle

\tableofcontents

\allowdisplaybreaks


\section{Introduction}

The notion of a Lie conformal algebra was introduced by Kac as an axiomatic description of the operator product expansion  of chiral fields in two-dimensional conformal field theory (\cite{Kac}). It turns out to be valuable tools in studying vertex algebras (\cite{BDK2,Bakalov,Kac,HL}), Poisson vertex algebras (\cite{BDK,LHS}) and Hamiltonian formalism in the theory of nonlinear evolution equations (\cite{BSK}). The structure theory (\cite{DAK}), representation theory (\cite{CK}) and cohomology theory (\cite{BKV,DK09}) of Lie conformal algebras have been well developed.

Motivated by the study of conformal analogues of Lie bialgebras, Liberati developed the theory of Lie conformal bialgebra in \cite{Lib}. The notion of conformal classical Yang-Baxter equation was also introduced to construct coboundary Lie conformal bialgebras. Explicitly, let $A$ be a Lie conformal algebra and $r=\sum_{i} a_i\otimes b_i\in A\otimes A$. Set $\partial^{\otimes^{3}}=\partial\otimes1\otimes1+1\otimes\partial\otimes1+1\otimes1\otimes\partial.$  The equation
\begin{eqnarray}
\nonumber \llbracket r,r\rrbracket&=&\sum_{i,j}([a_{i_{\mu}}a_{j}]\otimes b_{i}\otimes b_{j}| _{\mu=1\otimes \partial \otimes 1}-a_{i}\otimes[a_{j_{\mu}}b_{i}]\otimes b_{j}|_{\mu=1\otimes 1\otimes\partial}-a_{i}\otimes a_{j}\otimes[b_{j_{\mu}}b_{i}]|_{\mu=1\otimes\partial\otimes1})\\
 &=&0\quad\mbox{mod}\quad \partial^{\otimes^{3}}
\end{eqnarray}
is called {\bf conformal classical Yang-Baxter equation} in $A$. In particular, the skew-symmetric solutions of conformal classical Yang-Baxter equation give rise to Lie conformal bialgebras. Recently, the authors in \cite{HB} showed that $r$ is a skew-symmetric solution of the conformal classical Yang-Baxter equation in a finite Lie conformal algebra $A$ if and only if $r^{\sharp}_{0}=r^{\sharp}_{\lambda}|_{\lambda=0}$ satisfies
\begin{equation}
  [r^{\sharp}_{0}(\alpha)_{\lambda}r^{\sharp}_{0}(\beta)]=r^{\sharp}_{0}(\ad^*(r^{\sharp}_{0}(\alpha))_{\lambda}(\beta)-\ad^*(r^{\sharp}_{0}(\beta))_{-\lambda-\partial}(\alpha)), \quad \forall~ \alpha,\beta\in A^{\ast c},
\end{equation}
where $r^{\sharp}\in {\rm Chom}(A^{\ast c},A)$ is defined by
\begin{equation}
  r^{\sharp}_{\lambda}(\alpha)=\sum_i\langle\alpha,a_i\rangle_{-\lambda-\partial}b_i,\quad \forall~\alpha\in A^{\ast c},
\end{equation}
and $\ad^*$ is the coadjoint module over the $A$. More generally, the authors in \cite{HB} introduced the notion of $\huaO$-operator on a Lie conformal algebra.
\begin{defi}
    A $\Comp[\partial]$-module homomorphism $T:V\longrightarrow A$ is called an {\bf $\GRB$-operator} on an \LP pair  $(A,[\cdot_\lambda\cdot];\rho)$ if it satisfies
  \begin{equation}
  [T(u)_{\lambda}T(v)]=T(\rho(T(u))_{\lambda}(v)-\rho(T(v))_{-\lambda-\partial}(u)), \quad \forall~ u,v\in V.
\end{equation}
\end{defi}
Thus $r$ is a skew-symmetric solution of the conformal classical Yang-Baxter equation in a finite Lie conformal algebra $A$ if and only if $r^{\sharp}_{0}=r^{\sharp}_{\lambda}|_{\lambda=0}$ is an $\huaO$-operator on the \LP pair $(A;\ad^*)$. The importance of $\huaO$-operators on Lie conformal algebras is that  the skew-symmetric part of an $\huaO$-operator gives rise to a skew-symmetric solution of a conformal classical Yang-Baxter equation and the symmetric part of an $\huaO$-operator gives rise to a symmetry solution of a conformal $S$-equation, which plays an important role in the theory of left-symmetric conformal bialgebras (\cite{HL15}). In particular, when the module $\rho$ is the adjoint module, the $\huaO$-operator on a Lie conformal algebra is just the Rota-Baxter operator on the Lie conformal algebra. In the Lie algebra context, the notion of $\huaO$-operator  was introduced by Kupershmidt in \cite{Kuper1} as a tool to study the classical Yang-Baxter equation. See \cite{Bai1,HLS,LST,TBGS}  for more details and applications of $\huaO$-operators on Lie algebras.

Nijenhuis operators on Lie algebras have been introduced in the theory of integrable systems in the work of Magri, Gelfand and Dorfman (see the book \cite{Dorfman}). In the sense of the theory of deformations of Lie algebras (\cite{Nijenhuis}), Nijenhuis operators generate trivial deformations naturally. In this paper, we study the infinitesimal deformation of a Lie conformal algebra, which leads to the notion of Nijenhuis operator on the Lie conformal algebra and we show that it generates a trivial deformation of the Lie conformal algebra. See \cite{BKV} for more details on deformations of Lie conformal algebras.  We find that Nijenhuis
operators play an essential role in certain compatible structures
in terms of $\huaO$-operators. Roughly speaking, a Nijenhuis
operator ``connects'' two $\huaO$-operators on a Lie conformal algebra
whose any linear combination is still an $\mathcal O$-operator
(they are called ``compatible'' $\huaO$-operators) in some sense.

Poisson-Nijenhuis structures were defined by Magri and Morosi in 1984 in their study of completely integrable systems \cite{MM}. The importance of Poisson-Nijenhuis structures is that they produce bi-Hamiltonian systems. See \cite{Kosmann1,Kosmann2} for more details on Poisson-Nijenhuis structures. The aim of this paper is to study the theory of conformal analogues of Possion-Nijenhuis structures on Lie conformal algebras. We introduce the notion of an $\GRBN$-structure on a \LP pair (a Lie conformal algebra with a module), which consists of an $\huaO$-operator and a Nijenhuis structure on a \LP pair satisfying some compatibility conditions.  The notion of a Nijenhuis structure on a \LP pair consists of a Nijenhuis operator $N$ on $A$ and a $\Comp[\partial]$-linear map $S$ from $V$ to $V$ satisfying some compatibility conditions.  We point out that the introduction of the linear map ``$S$''  provides a general framework for the study of conformal analogues of Poisson-Nijenhuis structures.  See \cite{HLS,LBS} for more details about $\GRBN$-structures on Lie algebras and associative algebras. Furthermore, as an application of the theory of $\GRBN$-structures on \LP pairs, we introduce the notions of conformal $r$-matrix-Nijenhuis structures and symplectic-Nijenhuis structures on Lie conformal algebras, which are conformal analogues of $r$-matrix-Nijenhuis structures (\cite{Kosmann1,rn}) and symplectic-Nijenhuis structures on Lie algebras (\cite{Dorfman,Kosmann2}).

The paper is organized as follows. In Section \ref{sec:prel}, we recall the module and cohomology theory of Lie conformal algebras. In Section \ref{sec:Nijenhuis operators}, we study the infinitesimal deformations of Lie conformal algebras and \LP pairs. Then we introduce the notions of Nijenhuis operator on a Lie conformal algebra and Nijenhuis structure on a \LP pair, which give trivial deformations of the Lie conformal algebra and the \LP pair, respectively. In Section \ref{sec:ON-structures}, we add compatibility conditions between an $\GRB$-operator and a Nijenhuis structure to define the notion of $\GRBN$-structure on a \LP pair. Then we study compatible $\GRB$-operators on \LP pairs, and show that a Nijenhuis operator connects two compatible $\GRB$-operators. We also show that an $\GRBN$-structure gives rise to a hierarchy of pairwise compatible $\GRB$-operators. In Section \ref{sec:rn}, we first introduce the notion of conformal $r$-matrix-Nijenhuis structure on a Lie conformal algebra $A$, which consists of a skew-symmetric solution $r$ of the conformal classical Yang-Baxter equation and a Nijenhuis operator $N$ on $A$ such that some compatibility conditions are satisfied. We show that conformal $r$-matrix-Nijenhuis structure gives rise to a hierarchy of pairwise compatible $r$-matrices.  Furthermore, we introduce the notion of symplectic-Nijenhuis structure on a Lie conformal algebra $A$, which consists of a symplectic structure $\omega$ and a Nijenhuis operator $N$ on $A$ such that some compatibility conditions are satisfied.  The relations between $\GRBN$-structures on \LP pairs, conformal $r$-matrix-Nijenhuis structures and symplectic-Nijenhuis structures on Lie conformal algebras are given.

\section{Preliminaries on Lie conformal algebra}\label{sec:prel}
\begin{defi}
A {\bf Lie conformal algebra} $A$ is a $\Comp[\partial]$-module endowed with a $\Comp$-bilinear map $A\times A\rightarrow A[\lambda]$, denoted by $a\times b\rightarrow [a_{\lambda}b]$, satisfying
\begin{eqnarray}
\label{eq:Lie conformal1} [\partial a_{\lambda}b]&=&-\lambda[a_{\lambda}b],\\
\label{eq:Lie conformal2}  {[a_{\lambda}b]}&=&-[b_{-\lambda-\partial}a],\\
 \label{eq:Lie conformal3} {[a_{\lambda}[b_{\mu}c]]}&=&[[a_{\lambda}b]_{\lambda+\mu}c]+[b_{\mu}[a_{\lambda}c]],\quad\forall~ a,b,c\in A.
\end{eqnarray}
A Lie conformal algebra is called {\bf finite} if it is finitely generated as a $\Comp[\partial]$-module.
\end{defi}

The ordinary examples of finite Lie conformal algebra is the Virasoro and current Lie conformal algebras.
\begin{ex}
  The Virasoro Lie conformal algebra ${\rm Vir}$ is the simplest nontrivial example of Lie  conformal algebras, which is defined by
  $${\rm Vir}=\Comp[\partial]a,\quad [a_\lambda a]=(\partial+2\lambda)a.$$
\end{ex}

\begin{ex}
  Let $(\g,[\cdot,\cdot])$ be a Lie algebra. The current Lie conformal algebra associated with $\g$ is defined by
  $${\rm Cur}~\g=\Comp[\partial]\otimes\g,\quad [a_\lambda b]=[a,b],\quad \forall~a,b\in\g.$$
\end{ex}

In the following, we recall a class of Lie conformal algebras named quadratic Lie conformal algebras.
\begin{defi}
  If a Lie conformal algebra $A=\Comp[\partial] V$ with a free $\Comp[\partial]$-module and the $\lambda$-bracket is of the following form:
  $$[a_\lambda b]=\partial u+\lambda v+w,\quad \forall~a,b\in V,$$
  where $u,v,w\in V$, then $A$ is  called a {\bf quadratic Lie conformal algebra}.
\end{defi}

Recall that a {\bf Novikov algebra} is a pair $(A,\circ)$, where $A$ is a vector space and $\circ:A\otimes A\rightarrow A$ is a bilinear multiplication satisfying that for all $a,b,c\in A$,
\begin{eqnarray}
\label{eq:Nov1}  (a\circ b)\circ c&=&(a\circ c)\circ b,\\
\label{eq:Nov2}  (a\circ b)\circ c-a\circ (b\circ c)&=&(b\circ a)\circ c-b\circ (a\circ c).
\end{eqnarray}

Recall that a {\bf Gel'fand-Dorfman bialgebra} is a triple $(A,[\cdot,\cdot],\circ)$ such that $(A,[\cdot,\cdot])$ is a Lie algebra, $(A,\circ)$ is a Novikov algebra and they satisfy the following condition:
\begin{equation}\label{eq:GD condition}
  [a\circ b,c]+[a,b]\circ c-a\circ[b,c]-[a\circ c,b]-[a,c]\circ b=0.
\end{equation}

\emptycomment{Assume that $(A,\circ)$ is a  Novikov algebra. Then $(A,[\cdot,\cdot],\circ)$  is a  Gel'fand-Dorfman bialgebra, where the Lie bracket is given by
$$[a,b]=a\circ b-b\circ a,\quad\forall~a,b\in A.$$
It is obvious that }

\begin{pro}{\rm(\cite{Gel,XP})}
  A Lie conformal algebra $A=\Comp[\partial] V$ is quadratic if and only if $V$ is a Gel'fand-Dorfman bialgebra, where the correspondence is given as follows:
  $$[a_\lambda b]=\partial(b\circ a)+\lambda(a\ast b)+[b,a],$$
  where $a\ast b=a\circ b+b\circ a$ for $a,b\in V$.
\end{pro}
It is not hard to see that  Virasoro Lie conformal algebra is a quadratic Lie conformal algebra, which is corresponding to the Gel'fand-Dorfman bialgebra $(A,[\cdot,\cdot],\circ)$   induced by a  Novikov algebra $(V=\Comp a,\circ)$ with the multiplication $a\circ a=a$.

\begin{defi}
A {\bf module} $V$ over a Lie conformal algebra $A$ is a $\Comp [\partial]$-module endowed with a $\Comp$-bilinear map $A\times V\rightarrow V[\lambda]$, $(a,v)\rightarrow a_{\lambda}v$,satisfying the following axioms:
\begin{eqnarray}
  (\partial a)_{\lambda}v&=&-\lambda a_{\lambda}v,\\
  {a_{\lambda}(\partial v)}&=&(\partial+\lambda)a_{\lambda}v,\\
  {[a_{\lambda}b]_{\lambda+\mu}v}&=&a_{\lambda}(b_{\mu}v)-b_{\mu}(a_{\lambda}v),\quad\forall~ a,b\in A,v\in V.
\end{eqnarray}
An $A$-module $V$ is called {\bf finite} if it is finitely generated as a $\Comp [\partial]$-module.
\end{defi}

Throughout this paper, we mainly deal with $\Comp [\partial]$-modules which are finitely generated.

\begin{defi}
Let $U$ and $V$ be two $\Comp [\partial]$-modules. A {\bf conformal linear map} from $U$ to $V$ is a  $\Comp$-linear map $a: U\rightarrow V[\lambda]$, denoted by $a_{\lambda}: U\rightarrow V$, such that
$$[\partial, a_{\lambda}]=-\lambda a_{\lambda}.$$
 Denote the $\Comp$-vector space of all such maps by ${\rm Chom}(U,V)$.
\end{defi}
The space of conformal linear maps ${\rm Chom}(U,V)$ has a canonical structure of a $\Comp [\partial]$-module:
\begin{equation*}
  (\partial a)_{\lambda}=-\lambda a_{\lambda}.
\end{equation*}

Define the {\bf conformal dual} of a $\Comp[\partial]$-module $U$ as $U^{*c}={\rm Chom}(U,\Comp)$, where $\Comp$ is viewed as the trivial $\Comp[\partial]$-module, that is
  $$ U^{*c}= \{a:U\rightarrow \Comp[\lambda] \mid  \mbox{$a$ is $\Comp$-linear and } a_{\lambda}(\partial b)=\lambda a_{\lambda}b   \}.$$

Let $U$ and $V$ be finite modules over a Lie conformal algebra $A$. Then, the $\Comp[\partial]$-module ${\rm Chom}(U,V)$ has an A-module structure defined by:
\begin{equation}\label{eq:lie conformal 7}
  (a_{\lambda}\varphi)_{\mu}u=a_{\lambda}(\varphi_{\mu-\lambda}u)-\varphi_{\mu-\lambda}(a_{\lambda}u), \forall~a\in A, \varphi \in {\rm Chom}(U,V),u\in U.
\end{equation}

Let $V$ be a $\Comp [\partial]$-module and denote by ${\rm gc}(V)={\rm Chom}(V,V)$. The $\Comp[\partial]$-module ${\rm gc}(V)$ has a canonical structure of an associative conformal algebra defined by:
\begin{equation}
  (a_{\lambda}b)_{\mu}v=a_{\lambda}(b_{\mu-\lambda}v),\quad \forall~a,b\in {\rm gc}(V), v\in V.
\end{equation}
 Therefore, ${\rm gc}(V)$ has a Lie conformal algebra structure given by:
\begin{equation}
  [a_{\lambda}b]_{\mu}v=a_{\lambda}(b_{\mu-\lambda}v)-b_{\mu-\lambda}(a_{\lambda}v),\quad \forall~a,b\in {\rm gc}(V), v\in V.
\end{equation}
We call Lie conformal algebra ${\rm gc}(V)$ the {\bf general Lie conformal algebra} of $V$.

Note that the structure of a finite module $V$ over a Lie conformal algebra $A$ is the same as a homomorphism of Lie conformal algebra $\rho:A\rightarrow{\rm gc}(V)$. It is obvious that $(\Comp;\rho=0)$ is a module over the Lie conformal algebra $A$, which we call the {\bf trivial module}.

Let $(V;\rho)$ be a module over a  Lie conformal algebra $A$. Define $\rho^{*}:A\rightarrow {\rm gc}(V^{*c})$  by
\begin{equation*}
  (\rho^{*}(a)_{\lambda}\varphi)_{\mu}u=-\varphi_{\mu-\lambda}(\rho(a)_{\lambda}u)
\end{equation*}
 for all $a\in A$, $\varphi\in V^{*c}$, $u\in V$. Then $(V^*;\rho^*)$ is a module over $A$.

 Assume that $A$ is a finite Lie conformal algebra. Define $\ad:A\lon{\rm gc}(A)$ by $\ad(a)_\lambda b=[a_{\lambda}b]$ for all $a,b\in A$. Then  $(A;\ad)$ is a module over $A$, which we call the {\bf adjoint module}. Furthermore,  $(A^*;\ad^*)$ is also a module over $A$, which we call the {\bf coadjoint module}

The following conclusion is well-known.
\begin{pro}
Let $A$ be a finite Lie conformal algebra and let $(V;\rho)$ be a finite module over $A$. Then $A\oplus V$
is endowed with a $\Comp[\partial]$-module structure given by:
\begin{equation*}
  \partial(a+v)=\partial a+\partial v, \forall a\in A,v\in V.
\end{equation*}
Hence, the $\Comp[\partial]$-module $A\oplus V$ is endowed with a Lie conformal algebra structure as follows:
\begin{equation*}
  [(a+u)_{\lambda}(b+v)]=[a_{\lambda}b]+\rho(a)_{\lambda}v-\rho(b)_{-\lambda-\partial}u ,\quad \forall~a,b\in A,u,v\in V.
\end{equation*}
 This Lie conformal algebra is called the {\bf semi-direct product} of $A$ and $V$, denoted by $A\ltimes_{\rho}V$.
\end{pro}

 Let $A$ and $V$ be $\Comp[\partial]$-modules. For $k\geq 1 $, a {\bf $k$-$\lambda$-bracket} on $A$ with coefficients in $V$ is a $\Comp$-linear map $c:A^{\otimes k}\rightarrow \Comp[\lambda_{1}...\lambda_{k-1}]\otimes V $, denoted by
$$a_{1}\otimes...\otimes a_{k}\mapsto \{a_{1\lambda_{1}}...a_{k-1 \lambda_{k-1}}a_{k}\}_{c}$$
satisfying the following conditions:
\begin{eqnarray}
\label{eq:deformCA}  \{a_{1\lambda_{1}}...(\partial a_{i})_{\lambda_{i}}...a_{k-1 \lambda_{k-1}}a_{k}\}_{c}&=&-\lambda_{i}\{a_{1\lambda_{1}}...a_{k-1 \lambda_{k-1}}a_{k}\}_{c},\quad 1\leq i\leq k-1,\\
\label{eq:deforLA}\{a_{1\lambda_{1}}...a_{k-1 \lambda_{k-1}}(\partial a_{k})\}_{c}&=&(\lambda_{1}+...+\lambda_{k-1}+\partial)\{a_{1\lambda_{1}}...a_{k-1\lambda_{k-1}}a_{k}\}_{c},\\
\{a_{1\lambda_{1}}...a_{k-1 \lambda_{k-1}}a_{k}\}_{c}&=&{\rm sign}(\sigma)\{a_{\sigma(1)_{\lambda_{\sigma(1)}}}...a_{\sigma(k-1)_{\lambda_{\sigma(k-1)}}}a_{\sigma(k)}\}_{c}|_{\lambda_{k}\mapsto\lambda_{k}^{+}},
\end{eqnarray}
where the notation $\lambda_{k}\mapsto\lambda_{k}^{+}$ means that $\lambda_{k}$ is replaced by $\lambda_{k}^{+}=-\sum_{j=1}^{k-1}\lambda_{j}-\partial$, if it occurs, and $\partial ^{V}$ is moved to the left.

\emptycomment{\begin{rmk}
A structure of a Lie conformal algebra on $A$ is a $2-\lambda-bracket$ on $A$ with coefficients in $A$,satisfying the Jacobi identity.
\end{rmk}}

The cohomology complex for a Lie conformal algebra $A$ with a module $V$ in the language of $\lambda$-brackets is given as follows (\cite{DK09}). We let $C^{0}(A,V)=V/\partial V$ and for $k\geq1$, we denote by $C^{k}(A,V)$ the space of all $k$-$\lambda$-brackets on $A$ with coefficients in $V$. Define $C^{\bullet}(A,V)=\oplus _{k\in\Nat}C^{k}(A,V)$. The corresponding Lie conformal algebra coboundary operator $\dM:C^k(A,V)\longrightarrow C^{k+1}(A,V)$ is given by
\begin{equation*}
  \begin{split}
     \{a_{1\lambda_{1}}\cdots a_{k\lambda_{k}}a_{k+1}\}_{\dM c}&=\sum_{i=1}^{k}(-1)^{i+1}a_{i\lambda_{i}}\{a_{1\lambda_{1}}\cdots\check{i}\cdots a_{k\lambda_{k}}a_{k+1}\}_{c}\\ &+\sum_{i,j=1,i<j}^{k}(-1)^{k+i+j+1}\{a_{1\lambda_{1}}\cdots\check{i}\cdots\check{j}\cdots a_{k\lambda_{k}}a_{k+1_{\lambda_{k+1}^{+}}}[a_{i\lambda_{i}}a_{j}]\}_{c}\\
      &+(-1)^{k}a_{k+1_{\lambda_{k+1}^{+}}}\{a_{1\lambda_{1}}\cdots a_{k-1_{\lambda_{k-1}}}a_{k}\}_{c}\\
      &+\sum_{i=1}^{k}(-1)^{i}\{a_{1\lambda_{1}}\cdots\check{i}\cdots a_{k\lambda_{k}}[a_{i\lambda_{i}}a_{k+1}]\}_{c},
  \end{split}
\end{equation*}
where $c\in C^k(A,V) $, $\lambda_{k+1}^{+}=-\sum_{j=1}^{k}\lambda_{j}-\partial$ and $a_1,\cdots,a_{k+1}\in A$. We denote by $H^{\bullet}(A,V)=\oplus _{k\in\Nat}H^{k}(A,V)$ the cohomology of the Lie conformal algebra $A$ with coefficients in $V$. In particular, we use the symbol $\dM^T$ to refer the coboundary operator   associated to
the trivial module.

\section{Nijenhuis operators and Nijenhuis structures on Lie conformal algebras}\label{sec:Nijenhuis operators}
In this section, we study infinitesimal deformations of a Lie conformal algebra and an \LP pair. The trivial deformations of the Lie conformal algebra and the \LP pair lead to the notion of Nijenhuis operator on a Lie conformal algebra and Nijenhuis structure on an \LP pair, respectively.
\subsection{Nijenhuis operators on Lie conformal algebras}
Let $A$ be a Lie conformal algebra with $\lambda$-bracket $[\cdot_{\lambda}\cdot]$ and $\omega:A\otimes A\rightarrow A[\lambda]$. Consider a $t$-parameterized family of bilinear $\lambda$-multiplications
\begin{equation}
 [a_{\lambda}b]_{t}=[a_{\lambda}b]+t\{a_{\lambda}b\}_{\omega},\quad\forall~a,b\in A.
\end{equation}
If  $(A,[\cdot_{\lambda}\cdot]_t)$ is a Lie conformal algebra for all $t$, we say
that $\omega$ generates a {\bf infinitesimal
deformation} of $(A,[\cdot_{\lambda}\cdot])$.

By a direct calculation,  we can deduce that $(A,[\cdot_{\lambda}\cdot]_t)$ is a Lie conformal algebra  for any $t\in \mathbb \Real$ if and only if
\begin{eqnarray}
 \label{eq:lie conformal13}\{\partial a_{\lambda}b\}_{\omega}&=&-\lambda\{a_{\lambda}b\}_{\omega},\\
  \label{eq:lie conformal14}\{a_{\lambda}b\}_{\omega}&=&-\{b_{-\lambda-\partial}a\}_{\omega},\\
 \label{eq:lie conformal15}
 \{a_{\lambda}\{b_{\mu}c\}_{\omega}\}_{\omega}-\{b_{\mu}\{a_{\lambda}c\}_{\omega}\}_{\omega}&=&\{\{a_{\lambda}b\}_{\omega_{\lambda+\mu}}c\}_{\omega},
\end{eqnarray}
and
\begin{equation}
   \label{eq:lie conformal16}\{[a_{\lambda}b]_{\lambda+\mu}c\}_{\omega}+[\{a_{\lambda}b\}_{\omega_{\lambda+\mu}}c]-\{a_{\lambda}[b_{\mu}c]\}_{\omega}
 -[a_{\lambda}\{b_{\mu}c\}_{\omega}]+\{b_{\mu}[a_{\lambda}c]\}_{\omega}+[b_{\mu}\{a_{\lambda}c\}_{\omega}]=0.
\end{equation}

By \eqref{eq:lie conformal13} and  \eqref{eq:lie conformal14}, we  have
\begin{eqnarray}\label{eq:lie conformal13b}
\{ a_{\lambda}(\partial b)\}_{\omega}&=&(\lambda+\partial)\{a_{\lambda}b\}_{\omega}.
\end{eqnarray}

We see that  \eqref{eq:lie conformal13}-\eqref{eq:lie conformal15} mean that $(A,\{\cdot_{\lambda}\cdot\}_\omega)$ is a Lie conformal algebra. \eqref{eq:lie conformal16} means that $\omega$ is a $2$-cocycle of the Lie conformal algebra $A$ with coefficients in the module $(A;\ad)$. Thus we have
\begin{pro}
  If  $\omega$ generates a  infinitesimal
deformation of the Lie conformal algebra $(A,[\cdot_{\lambda}\cdot])$, then $\omega$ is a $2$-cocycle of $(A,[\cdot_{\lambda}\cdot])$.
\end{pro}

\begin{defi}
  Two infinitesimal deformations $A_t=(A,[\cdot_{\lambda}\cdot]_t)$ and $A'_t=(A,[\cdot_{\lambda}\cdot]'_t)$  of a
Lie conformal algebra $(A,[\cdot_{\lambda}\cdot])$, which are generated by $\omega$
and $\omega'$  respectively, are said to be {\bf
equivalent} if there exists a family of Lie conformal algebra
homomorphisms ${\Id}+tN:A_t\longrightarrow A'_t$. A deformation
is said to be {\bf trivial} if there exist a family of Lie conformal algebra homomorphisms ${\Id}+tN:A_t\longrightarrow (A,[\cdot_{\lambda}\cdot])$.
\end{defi}

By a direct calculation, $A_t$ and $A'_t$  are
equivalent infinitesimal deformations  if and only if
\begin{eqnarray}
N\circ \partial&=&\partial\circ N,\label{eq:equivalent1}\\
\{a_\lambda b\}_\omega-\{a_\lambda b\}_{\omega'}&=&[N(a)_\lambda b]+[a_\lambda N(b)]-N[a_\lambda b],\label{eq:equivalent2}\\
N\{a_\lambda b\}_\omega&=&\{N(a)_\lambda b\}_{\omega'}+\{a_\lambda N(b)\}_{\omega'}+[N (a)_\lambda N(b)],\label{eq:equivalent3}\\
\{N(a)_\lambda N(b)\}_{\omega'}&=&0.\label{eq:equivalent4}
\end{eqnarray}

It is obvious that \eqref{eq:equivalent1} means that $N$ is a $1$-cochain and $(\ref{eq:equivalent2})$ means that $\{a_\lambda b\}_\omega-\{a_\lambda b\}_{\omega'}=\{a_\lambda b\}_{\dM N}$. We summarize the above discussions into the following conclusion:
\begin{pro}\label{thm:deformation}
If two infinitesimal deformations $A_t=(A,[\cdot_{\lambda}\cdot]_t)$ and $A'_t=(A,[\cdot_{\lambda}\cdot]'_t)$  of a
Lie conformal algebra $(A,[\cdot_{\lambda}\cdot])$ generated by $\omega $ and
$\omega'$  respectively  are equivalent,  then
$\omega$ and $\omega'$ are in the same cohomology class in $H^{2}(A,A)$.
\end{pro}

Now we consider  trivial deformations of a Lie conformal algebra $(A,[\cdot_{\lambda}\cdot])$. Then \eqref{eq:equivalent1}-\eqref{eq:equivalent4} reduce to
\begin{eqnarray}
N\circ \partial&=&\partial\circ N,\label{Nij1}\\
\{a_\lambda b\}_\omega&=&[N(a)_\lambda b]+[a_\lambda N(b)]-N[a_\lambda b],\label{Nij2}\\
N\{a_\lambda b\}_\omega&=&[N (a)_\lambda N(b)].\label{Nij3}
\end{eqnarray}

 It follows from \eqref{Nij2} and \eqref{Nij3} that $N$ should satisfy the following condition
 \begin{equation}
   \label{eq:Nijenhuis condition} N\big([N(a)_{\lambda}(b)]+[a_{\lambda}N(b)]-N[a_{\lambda}b]\big)=[N(a)_{\lambda}N(b)].
 \end{equation}

The following definition first appears in the study of Nijenhuis operators on associative algebras.
 \begin{defi}{\rm(\cite{Yuan})}
 Let $(A,[\cdot_{\lambda}\cdot])$ be a Lie conformal algebra. A $\Comp[\partial]$-module homomorphism $N:A\rightarrow A$ is called a {\bf Nijenhuis operator} if \eqref{eq:Nijenhuis condition} holds.
 \end{defi}

\begin{thm}\label{thm:lie conformal3.4}
Let $N:A\rightarrow A$ be a Nijenhuis operator. A infinitesimal deformation of $A$ can be obtained by putting
$$ \{a_{\lambda}b\}_{\omega}=[N(a)_{\lambda}b]+[a_{\lambda}N(b)]-N[a_{\lambda}b],\quad\forall~a,b\in A.$$
Moreover, this deformation is trivial.
\end{thm}
\begin{proof}
By \eqref{eq:Lie conformal1} and $N(\partial a)=\partial  N(a)$, we have
\begin{equation*}
 \begin{split}
  \{\partial a_{\lambda} b\}_{\omega}&=[\partial a_{\lambda}N(b)]+[N(\partial a)_{\lambda}b]-N^{2}[\partial a_{\lambda}b]\\
  &=-\lambda[a_{\lambda}N(b)]-\lambda[N(a)_{\lambda}(b)+\lambda N^{2}[a_{\lambda}b]=-\lambda \{a_{\lambda}b\}_{\omega},
 \end{split}
\end{equation*}
which proves \eqref{eq:lie conformal13}.

By \eqref{eq:Lie conformal2}, we have
\begin{equation*}
  \begin{split}
     \{a_{\lambda}b\}_{\omega} &=[a_{\lambda}N(b)]+[N(a)_{\lambda}b]-N^{2}[a_{\lambda}b]\\
       &=-([b_{-\lambda-\partial}N(a)]-[N(b)_{-\lambda-\partial}a]-N^{2}[b_{-\lambda-\partial} a])\\
       &=-\{b_{-\lambda-\partial}a\}_{\omega},
   \end{split}
\end{equation*}
which proves \eqref{eq:lie conformal14}.

Since $\omega=dN$, we get $d\omega=d(dN)=0$ and thus \eqref{eq:lie conformal16} holds.

Next we  prove that the Jacobi identity holds for $\{\cdot_\lambda\cdot\}_\omega$. Put
\begin{equation*}
  J(a_1,a_2,a_{3})=\{a_{1\lambda_{1}}\{a_{2\lambda_{2}}a_{3}\}_{\omega}\}_{\omega}-\{a_{2\lambda_{2}}\{a_{1\lambda_{1}}a_{3}\}_{\omega}\}_{\omega}
  -\{\{a_{1\lambda_1}a_{2}\}_{\omega_{\lambda_{1}+\lambda_{2}}}a_{3}\}_{\omega}.
\end{equation*}
Substituting the expression of $\{\cdot_\lambda\cdot\}_\omega$ into this formula and using the property of Nijenhuis operator on $A$, we have
\begin{eqnarray*}
     J(a_1,a_2,a_{3})&=&[a_{1_{\lambda_{1}}}[Na_{2\lambda_{2}}Na_{3}]]+[Na_{1\lambda_{1}}[a_{2\lambda_{2}}Na_{3}]]-N[a_{1\lambda_{1}}[a_{2\lambda_{2}}Na_{3}]]\\ &&+[Na_{1\lambda_{1}}[Na_{2\lambda_{2}}a_{3}]]-N[a_{1\lambda_{1}}[Na_{2\lambda_{2}}a_{3}]]-[Na_{1\lambda_{1}}N[a_{2\lambda_{2}}a_{3}]]\\
     &&-[a_{2_{\lambda_{2}}}[Na_{1\lambda_{1}}Na_{3}]]-[Na_{2\lambda_{2}}[a_{1\lambda_{1}}Na_{3}]]+N[a_{2\lambda_{2}}[a_{1\lambda_{1}}Na_{3}]]-[Na_{2\lambda_{2}}[Na_{1\lambda_{1}}a_{3}]]\\
     &&+N[a_{1\lambda_{1}}N[a_{2\lambda_{2}}a_{3}]]+N[a_{2\lambda_{2}}[Na_{1\lambda_{1}}a_{3}]]+[Na_{2\lambda_{2}}N[a_{1\lambda_{1}}a_{3}]]\\
     &&-N[a_{2\lambda_{2}}N[a_{1\lambda_{1}}a_{3}]]-[[a_{1\lambda_{1}}Na_{2}]_{\lambda_{1}+\lambda_{2}}Na_{3}]-[[Na_{1\lambda_{1}}Na_{2}]_{\lambda_{1}+\lambda_{2}}a_{3}]\\
     &&+N[[a_{1\lambda_{1}}Na_{2}]_{\lambda_{1}+\lambda_{2}}a_{3}]-[[Na_{1\lambda_{1}}a_{2}]_{\lambda_{1}+\lambda_{2}}Na_{3}]+N[[Na_{1\lambda_{1}}a_{2}]_{\lambda_{1}+\lambda_{2}}a_{3}]\\
     &&+[N[a_{1\lambda_{1}}a_{2}]_{\lambda_{1}+\lambda_{2}}Na_{3}]-N[N[a_{1\lambda_{1}}a_{2}]_{\lambda_{1}+\lambda_{2}}a_{3}].
\end{eqnarray*}
Due to the validity of the Jacobi identity of $[\cdot_\lambda\cdot]$,  we have
\begin{eqnarray*}
      J(a_1,a_2,a_{3})&=&-N[a_{1\lambda_{1}}[a_{2\lambda_{2}}Na_{3}]]-N[a_{1\lambda_{1}}[Na_{2\lambda_{2}}a_{3}]]-[Na_{1\lambda_{1}}N[a_{2\lambda_{2}}a_{3}]]\\
      &&+N[a_{1\lambda_{1}}N[a_{2\lambda_{2}}a_{3}]]+N[a_{2\lambda_{2}}[a_{1\lambda_{1}}Na_{3}]]+N[a_{2\lambda_{2}}[Na_{1\lambda_{1}}a_{3}]]\\
      &&+[Na_{2\lambda_{2}}N[a_{1\lambda_{1}}a_{3}]]-N[a_{2\lambda_{2}}N[a_{1\lambda_{1}}a_{3}]]+N[[a_{1\lambda_{1}}Na_{2}]_{\lambda_{1}+\lambda_{2}}a_{3}]\\
      &&+N[[Na_{1\lambda_{1}}a_{2}]_{\lambda_{1}+\lambda_{2}}a_{3}]+[N[a_{1\lambda_{1}}a_{2}]_{\lambda_{1}+\lambda_{2}}Na_{3}]-N[N[a_{1\lambda_{1}}a_{2}]_{\lambda_{1}+\lambda_{2}}a_{3}].
\end{eqnarray*}
Since $N$ is a Nijenhuis operator,we have
\begin{eqnarray*}
  -{[Na_{1\lambda_{1}}N[a_{2\lambda_{2}}a_{3}]]}+N[a_{1\lambda_{1}}N[a_{2\lambda_{2}}a_{3}]]&=&-N[Na_{1\lambda_{1}}[a_{2\lambda_{2}}a_{3}]]+N^{2}[a_{1\lambda_{1}}[a_{2\lambda_{2}}a_{3}]],\\
  {[Na_{2\lambda_{2}}N[a_{1\lambda_{1}}a_{3}]]}-N[a_{2\lambda_{2}}N[a_{1\lambda_{1}}a_{3}]]&=&N[Na_{2\lambda_{2}}[a_{1\lambda_{1}}a_{3}]]+N^{2}[a_{2\lambda_{2}}[a_{1\lambda_{1}}a_{3}]],\\
  {[N[a_{1\lambda_{1}}a_{2}]_{\lambda_{1}+\lambda_{2}}Na_{3}]}-N[N[a_{1\lambda_{1}}a_{2}]_{\lambda_{1}+\lambda_{2}}a_{3}]&=&N[[a_{1\lambda_{1}}a_{2}]_{\lambda_{1}+\lambda_{2}}Na_{3}]-N^{2}[[a_{1\lambda_{1}}a_{2}]_{\lambda_{1}+\lambda_{2}}a_{3}].
\end{eqnarray*}
Again using the Jacobi identity of $[\cdot_\lambda\cdot]$,  we have
\begin{eqnarray*}
     J(a_{1\lambda_{1}}a_{2\lambda_{2}}a_{3})&=&-N\big([a_{1\lambda_{1}}[a_{2\lambda_{2}}Na_{3}]]-[a_{2\lambda_{2}}[a_{1\lambda_{1}}Na_{3}]]-[[a_{1\lambda_{1}}a_{2}]_{\lambda_{1}+\lambda_{2}}Na_{3}]\big)\\
       &&+[a_{1\lambda_{1}}[Na_{2\lambda_{2}}a_{3}]]-[a_{2\lambda_{2}}[Na_{1\lambda_{1}}a_{3}]]-[[a_{1\lambda_{1}}Na_{2}]_{\lambda_{1}+\lambda_{2}}a_{3}]\\
       &&+[Na_{1\lambda_{1}}[a_{2\lambda_{2}}a_{3}]]-[Na_{2\lambda_{2}}[a_{1\lambda_{1}}a_{3}]]-[[Na_{1\lambda_{1}}a_{2}]_{\lambda_{1}+\lambda_{2}}a_{3}]\\
       &&+N^{2}([a_{1\lambda_{1}}[a_{2\lambda_{2}}a_{3}]]-[a_{2\lambda_{2}}[a_{1\lambda_{1}}a_{3}]]-[[a_{1\lambda_{1}}a_{2}]_{\lambda_{1}+\lambda_{2}}a_{3}])\\
       &=&0.
\end{eqnarray*}
Thus $\omega$ generates a infinitesimal deformation of $A$.

Furthermore, \eqref{Nij1}-\eqref{Nij3} are satisfied and therefore  $\omega$ generates a trivial deformation of $A$.
\end{proof}

\begin{cor}\label{cor:Nijenhuis}
Let $(A,[\cdot_\lambda\cdot])$ be a Lie conformal algebra and $N$ a Nijenhuis operator on $A$. Then $(A,[\cdot_\lambda\cdot]_N)$ is a Lie conformal algebra and $N$ is a Lie conformal algebra homomorphism from $(A,[\cdot_\lambda\cdot]_N)$ to $(A,[\cdot_\lambda\cdot])$, where the deformed bracket $[\cdot_\lambda\cdot]_N$ is given by
\begin{equation}\label{eq:deformed bracket}
\{a_{\lambda}b\}_{N}=[N(a)_{\lambda}b]+[a_{\lambda}N(b)]-N[a_{\lambda}b],\quad\forall~a,b\in A.
\end{equation}
\end{cor}

\emptycomment{Similar to the properties of Nijenhuis operators on Lie algebras given in \cite{Kosmann1}, we also have
\begin{pro}\label{lem:Niejproperty}
  Let $(A,[\cdot_\lambda\cdot])$ be a Lie conformal algebra and $N$ a Nijenhuis operator on $A$. For all $k,l\in\Nat$,
  \begin{itemize}
\item[$\rm(i)$]$(A,[\cdot_\lambda\cdot]_{N^k})$ is a Lie conformal algebra;
\item[$\rm(ii)$]$N^l$ is also a Nijenhuis operator on the Lie conformal algebra $(A,[\cdot_\lambda\cdot]_{N^k})$;
\item[$\rm(iii)$]The Lie conformal algebras $(A,([\cdot_\lambda\cdot]_{N^k})_{N^l})$ and $(A,[\cdot_\lambda\cdot]_{N^{k+l}})$ coincide;
\item[$\rm(iv)$]The Lie conformal algebras $(A,[\cdot_\lambda\cdot]_{N^k})$ and $(A,[\cdot_\lambda\cdot]_{N^l})$ are
compatible, that is,
any linear combination of $(A,[\cdot_\lambda\cdot]_{N^k})$ and $(A,[\cdot_\lambda\cdot]_{N^l})$ still
makes $A$ into a Lie conformal algebra;
\item[$\rm(v)$]$N^l$ is a Lie conformal algebra homomorphism from $(A,[\cdot_\lambda\cdot]_{N^{k+l}})$ to $(A,[\cdot_\lambda\cdot]_{N^k})$.
  \end{itemize}
\end{pro}
\begin{proof}
  It follows by a straightforward calculation. We omit the details.
\end{proof}}

\begin{ex}
  Let $(A,[\cdot_\lambda\cdot])$ be a Lie conformal algebra,  which can be decomposed as a sum of two subalgebras $A=A_1\oplus A_2$. For any $k_1,k_2\in\Comp$, define $N:A\rightarrow A$ by
  $$N\mid_{A_i}=k_i\Id_{A_i},\quad i=1,2.$$
  Then $N$ is a Nijenhuis operator on the Lie conformal algebra $A$.
\end{ex}

Recall that a {\bf Nijenhuis operator} $N$ on a Lie algebra $(\g,[\cdot,\cdot])$ is a linear operator $N:\g\longrightarrow \g$ satisfying
$$N([N(x),y]+[x,N(y)]-N([x,y]))=[N(x),N(y)],\quad \forall~x,y\in\g.$$
Then $(\g,[\cdot,\cdot]_N)$ is a Lie algebra and $N$ is a morphism
from the Lie algebra $(\g,[\cdot,\cdot]_N)$ to
$(\g,[\cdot,\cdot])$, where the bracket  $[\cdot,\cdot]_N$ is
defined by
$$[x,y]_N=[N(x),y]+[x,N(y)]-N([x,y]),\;\;\forall x,y\in \g.$$

\begin{ex}
  Let $N$ be a Nijenhuis operator on  a Lie algebra $(\g,[\cdot,\cdot])$. Then $\tilde{N}$ defined by
  \begin{equation}\label{ex:Nij-current}
  \tilde{N}(f(\partial)a)=f(\partial)N(a),\quad \forall~f(\partial)\in\Comp[\partial],a\in \g
  \end{equation}
   is a Nijenhuis operator on the current Lie conformal algebra ${\rm Cur}~\g=\Comp[\partial]\otimes\g$.
\end{ex}

\begin{defi}
  A {\bf Nijenhuis operator} $N$ on a Novikov algebra $(A,\circ)$ is a linear operator $N:A\longrightarrow A$ satisfying
$$N(N(a)\circ b+a\circ N(b)-N(a\circ b))=N(a)\circ N(b),\quad \forall~a,b\in A.$$
\end{defi}
\begin{pro}\label{pro:Novikov-Nij}
  Let $N$ be a Nijenhuis operator on a Novikov algebra $(A,\circ)$. Then $(A,\circ_N)$ is also a Novikov algebra and furthermore, $N$ is a Novikov algebra homomorphism from $(A,\circ_N)$ to $(A,\circ)$.
\end{pro}
\begin{proof}
It was shown in \cite{WBLS} that $(A,\circ_N)$ is a left-symmetric algebra, which means that \eqref{eq:Nov2} holds for the operation $\circ_N$. Also, by the fact that $N$ is a   Nijenhuis operator on the Novikov algebra $(A,\circ)$, we have
\begin{eqnarray*}
 && (a\circ_N b)\circ_N c-(a\circ_N c)\circ_N c\\
  &=&\big((N(a)\circ N(b))\circ c-(N(a)\circ c)\circ N(b)\big)+\big((N(a)\circ b)\circ N(c)-(N(a)\circ N(c))\circ b\big)\\
  &&+\big((a\circ N(b))\circ N(c)-(a\circ N(c))\circ N(b)\big)-N\big((a\circ b)\circ N( c)-(a\circ N( c))\circ b\big)\\
  &&-N\big((N(a)\circ b)\circ  c-(N(a)\circ c)\circ b\big)-N\big((a\circ N(b))\circ  c-(a\circ  c)\circ N(b)\big)\\
  &&+N^2\big((a\circ b)\circ c- (a\circ c)\circ b \big).
\end{eqnarray*}
By \eqref{eq:Nov1}, the above equality implies that $(a\circ_N b)\circ_N c=(a\circ_N c)\circ_N c$ holds. Thus $(A,\circ_N)$ is also a Novikov algebra. The rest is direct.
\end{proof}

\begin{defi}
 Let $(A,[\cdot,\cdot],\circ)$  be a Gel'fand-Dorfman bialgebra. A linear map $N:A\rightarrow A$ is called a {\bf Nijenhuis operator} on the Gel'fand-Dorfman bialgebra $A$ if $N$ is both a Nijenhuis operator on the Lie algebra $(A,[\cdot,\cdot])$ and a Nijenhuis operator on the Novikov algebra $(A,\circ)$.
\end{defi}

\begin{pro}
  Let $N$ be a Nijenhuis operator on a Gel'fand-Dorfman bialgebra $(A,[\cdot,\cdot],\circ)$. Then $(A,[\cdot,\cdot]_N,\circ_N)$ is also a Gel'fand-Dorfman bialgebra and furthermore, $N$ is a Gel'fand-Dorfman bialgebra homomorphism from $(A,[\cdot,\cdot]_N,\circ_N)$ to $(A,[\cdot,\cdot],\circ)$.
\end{pro}
\begin{proof}
  Since $N$ is a Nijenhuis operator on the Lie algebra $(A,[\cdot,\cdot])$, $(A,[\cdot,\cdot]_N)$ is a Lie algebra. Similarly, $(A,\circ_N)$ is also a Novikov algebra. Furthermore, by the properties of Nijenhuis operator on Gel'fand-Dorfman bialgebra, by a direct calculation, we have
  \begin{eqnarray*}
     &&[a\circ_N b,c]_N+[a,b]_N\circ_N c-a\circ_N[b,c]_N-[a\circ_N c,b]_N-[a,c]_N\circ_N b\\
     &=&\big([N(a)\circ N(b),c]+[N(a),N(b)]\circ c-N(a)\circ[N(b),c]-[N(a)\circ c,N(b)]-[N(a),c]\circ N(b)\big)\\
     &&+\big([N(a)\circ b,N(c)]+[N(a),b]\circ N(c)-N(a)\circ[b,N(c)]-[N(a)\circ N(c),b]-[N(a),N(c)]\circ b\big)\\
     &&+\big([a\circ N(b),N(c)]+[a,N(b)]\circ N(c)-a\circ[N(b),N(c)]-[a\circ N(c),N(b)]-[a,N(c)]\circ N(b)\big)\\
     &&-N\big([N(a)\circ b,c]+[N(a),b]\circ c-N(a)\circ[b,c]-[N(a)\circ c,b]-[N(a),c]\circ b\big)\\
     &&-N\big([a\circ N(b),c]+[a,N(b)]\circ c-a\circ[N(b),c]-[a\circ c,N(b)]-[a,c]\circ N(b)\big)\\
     &&-N\big([a\circ b,N(c)]+[a,b]\circ N(c)-a\circ[b,N(c)]-[a\circ N(c),b]-[a,N(c)]\circ b\big)\\
     &&+N^2\big([a\circ b,c]+[a,b]\circ c-a\circ[b,c]-[a\circ c,b]-[a,c]\circ b\big).
  \end{eqnarray*}
  By \eqref{eq:GD condition}, the above equality implies that
  $$[a\circ_N b,c]_N+[a,b]_N\circ_N c-a\circ_N[b,c]_N-[a\circ_N c,b]_N-[a,c]_N\circ_N b=0.$$
  Thus $(A,[\cdot,\cdot]_N,\circ_N)$ is a Gel'fand-Dorfman bialgebra.
\end{proof}

\begin{pro}
  Let $N$ be a Nijenhuis operator on a Gel'fand-Dorfman bialgebra $(A,[\cdot,\cdot],\circ)$. Then $\tilde{N}$ defined by
  $$\tilde{N}(f(\partial)a)=f(\partial)N(a)$$
  is a Nijenhuis operator on the quadratic Lie conformal algebra $A=\Comp[\partial] V$ with the following $\lambda$-bracket
  $$[a_\lambda b]=\partial(b\circ a)+\lambda(a\ast b)+[b,a].$$
  Furthermore, $(A,[\cdot_\lambda \cdot]_N)$ is also a quadratic Lie conformal algebra with the $\lambda$-bracket $[\cdot_\lambda \cdot]_{\tilde{N}} $ given by
  $$[a_\lambda b]_{\tilde{N}}=\partial(b\circ_{\tilde{N}} a)+\lambda(a\ast_{\tilde{N}} b)+[b,a]_{\tilde{N}},$$
  where $a\ast_{\tilde{N}} b=a\circ_{\tilde{N}} b+b\circ_{\tilde{N}} a$ and this quadratic Lie conformal algebra corresponds to the Gel'fand-Dorfman bialgebra $(A,[\cdot,\cdot]_N,\circ_N)$.
\end{pro}
\begin{proof}
  It follows by a direct calculation.
\end{proof}

\begin{ex}
 The Virasoro Lie conformal algebra is a quadratic Lie conformal algebra, which is induced by a  Novikov algebra $(V=\Comp a,\circ)$ with the multiplication $a\circ a=a$. It is straightforward to check that for any $k\in \Comp$, $N(a)=ka$ is a Nijenhuis operator on the Novikov algebra $(V=\Comp a,\circ)$. Thus $\tilde{N}$ defined by
  $$\tilde{N}(f(\partial)a)=kf(\partial)a$$
  is a Nijenhuis operator on the Virasoro Lie conformal algebra $\Comp[\partial] a$. Furthermore, any Nijenhuis operator on the Virasoro Lie conformal algebra $\Comp[\partial] a$ is of this form.
\end{ex}

The following example says that not all Nijenhuis operators on quadratic Lie conformal algebra induced by the Nijenhuis operators on Gel'fand-Dorfman bialgebra.
\begin{ex}\label{ex:N-operator rank2}
  Let $A=\Comp[\partial]a\oplus \Comp[\partial]b$ be a quadratic Lie conformal algebra of rank $2$ with the $\lambda$-bracket given by:
  $$[a_\lambda a]=(\partial+2\lambda)a,\quad [a_\lambda b]=(\partial+\lambda)b,\quad [b_\lambda b]=0,$$
  which is induced by a  Novikov algebra $(V=\Comp a\oplus \Comp b,\circ)$ with the multiplication
  $$a\circ a=a,\quad a\circ b=0,\quad b\circ a=b,\quad b\circ b=0.$$
  Then $N$ given by
  $$N(a)=f(\partial)b,\quad N(b)=0,\quad f(\partial)\in \Comp[\partial]$$
  is a Nijenhuis operator on the Lie conformal algebra $A=\Comp[\partial]a\oplus \Comp[\partial]b$, which is not induced by any of the Nijenhuis operator on the Novikov algebra $(V=\Comp a\oplus \Comp b,\circ)$.
\end{ex}

\subsection{Nijenhuis structures on \LP pairs }\label{sec:Nijenhuis structures}
\begin{defi}
An {\bf \LP pair} consists of a Lie conformal algebra $(A,[\cdot_\lambda\cdot])$ and a finite module $(V;\rho)$ over $A$. We denote an \LP pair by $(A,[\cdot_\lambda\cdot];\rho)$ or simply by $(A;\rho)$.
\end{defi}

Let $(A;\rho)$ be an \LP pair. Let $ \omega:A\otimes A \rightarrow A[\lambda] $ and $\varpi:A\rightarrow {\rm gc}(V)$ be linear maps. Consider a $t$-parametrized family of bracket operations and linear maps
\begin{eqnarray*}
[a_{\lambda}b]_{t}&=&[a_{\lambda}b]+t\omega_{\lambda}(a,b),\\
 \rho_{t}(a)_{\mu}(v)&=&\rho(a)_{\mu}(v)+t\varpi(a)_{\mu}(v).
\end{eqnarray*}
If $(A,[\cdot_{\lambda}\cdot]_{t})$ is a Lie conformal algebra and $(V;\rho_{t})$ is a module for any $t\in \Real$, we say that $(\omega,\varpi)$ generates a {\bf infinitesimal deformation} of the \LP pair $(A;\rho)$. We denote a infinitesimal deformation of the \LP pair $(A,[\cdot_\lambda\cdot];\rho)$ by $(A,[\cdot_{\lambda}\cdot]_{t};\rho_{t})$.

By a direct calculation,  we can deduce that $(A,[\cdot_{\lambda}\cdot]_t;\rho_t)$ is an \LP pair  for any $t\in \mathbb \Real$ if and only if $(\omega,\varpi)$ satisfies \eqref{eq:lie conformal13}-\eqref{eq:lie conformal16} and
\begin{eqnarray}
\label{eq:lie conformal26}\rho(\{a_{\lambda}b\}_{\omega})_{\mu}(v)&+&\rho(b)_{\mu-\lambda}\varpi(a)_{\lambda}(v)-\rho(a)_{\lambda}\varpi(b)_{\mu-\lambda}(v)\\
\nonumber&=&\varpi(a)_{\lambda}\rho(b)_{\mu-\lambda}(v)-\varpi[a_{\lambda}b]_{\mu}(v)-\varpi(b)_{\mu-\lambda}\rho(a)_{\lambda}(v),\\
\label{eq:lie conformal27}  \varpi(\{a_{\lambda}b\}_{\omega})_{\mu}(v)&=&
 \varpi(a)_{\lambda}(\varpi(b)_{\mu-\lambda}(v))-\varpi(b)_{\mu-\lambda}(\varpi(a)_{\lambda}(v)),
\end{eqnarray}
where $a,b\in A$ and $v\in V$.

\begin{defi}
Two infinitesimal deformations $(A,[\cdot_{\lambda}\cdot]_{t};\rho_{t})$ and $(A',[\cdot_{\lambda}\cdot];\rho_t')$ of an \LP pair $(A,[\cdot_\lambda\cdot];\rho)$ are {\bf equivalent} if for $N:A\longrightarrow A$ and $S:V\longrightarrow V$, there exists an isomorphism $({\Id}_{A}+tN,{\Id}_{V}+tS)$ from  $(A',[\cdot_{\lambda}\cdot];\rho_t')$ to $(A,[-_{\lambda}-]_{t},\rho_{t})$, i.e. for all $a,b\in A$ and $v\in V$,
\begin{eqnarray*}
({\Id}_{A}+tN)[a_{\lambda}b]'_{t}&=&[({\Id}_{A}+tN)(a)_{\lambda}({\Id}_{A}+tN)(b)],\\
({\Id}_{V}+tS)\rho'_{t}(a)_{\mu}(v)&=&\rho(({\Id}_{A}+tN)(a))_{\mu}({\Id}_{V}+tS)(v).
\end{eqnarray*}
A infinitesimal deformation of a \LP pair $(A,[\cdot_\lambda\cdot];\rho)$ is said to be {\bf trivial} if it is equivalent to $(A,[-_{\lambda}-];\rho)$.
\end{defi}

It is straightforward to check that $(A,[-_{\lambda}-]_{t};\rho_{t})$ is a trivial deformation if and only if
\begin{eqnarray}
 \label{eq:lie conformalcom1}N\circ \partial&=&\partial\circ N,\\
 \label{eq:lie conformalcom2}S\circ \partial&=&\partial\circ S,\\
  \label{eq:lie conformal28}\{a_{\lambda}b\}_{\omega}&=&[a_{\lambda}Nb]+[Na_{\lambda}b]-N[a_{\lambda}b],\\
    \label{eq:lie conformal29} N\{a_{\lambda}b\}_{\omega}&=&[Na_{\lambda}Nb],\\
  \label{eq:lie conformal30} \varpi(a)_{\mu}(v)&=&\rho(N(a))_{\mu}(v)+\rho(a)_{\mu} S(v)-S\rho(a)_{\mu}(v),\\
   \label{eq:lie conformal31}\rho(N(a))_{\mu} S(v)&=&S  \varpi(a)_{\mu}(v).
\end{eqnarray}
 It follows from \eqref{eq:lie conformalcom1}, \eqref{eq:lie conformal28} and \eqref{eq:lie conformal29} that $N$ is a Nijenhuis operator on the Lie conformal algebra $(A,[\cdot_\lambda\cdot])$.
 It follows from \eqref{eq:lie conformal30} and \eqref{eq:lie conformal31} that
 \begin{equation}
   \label{eq:lie conformal33}\rho(N(a))_{\mu}(S(v))=S(\rho(N(a))_{\mu}(v))+S(\rho(a)_{\mu}S(v))-S^{2}(\rho(a)_{\mu}(v)).
 \end{equation}

In a trivial infinitesimal deformation of an \LP pair $(A;\rho)$, $N$ is a Nijenhuis operator on the Lie conformal algebra $A$ and conditions \eqref{eq:lie conformalcom2} and \eqref{eq:lie conformal33} hold. In fact, the converse is also true.
\begin{thm}\label{thm:trivial deform}
 If $N$ is a Nijenhuis operator on the Lie conformal algebra $A$ and  $S$ satisfies \eqref{eq:lie  conformalcom2} and \eqref{eq:lie conformal33}, then a deformation of the \LP pair $(A;\rho)$ can be obtained by putting
  \begin{eqnarray}
   \{a_{\lambda}b\}_{\omega}&=&[a_{\lambda}Nb]+[Na_{\lambda}b]-N[a_{\lambda}b],\\
  \varpi(a)_{\mu}(v)&=& \rho(N(a))_{\mu}(v)+\rho(a)_{\mu}S(v)-S\rho(a)_{\mu}(v),\quad\forall~a,b\in A,v\in V.
 \end{eqnarray}
Furthermore, this deformation is trivial.
\end{thm}
\begin{proof}
 It is a straightforward computation. We omit the details
\end{proof}

It is not hard to check that
\begin{pro}
Let $(A;\rho)$ be an \LP pair. Then $N$ is a Nijenhuis operator on the Lie conformal algebra $A$ and $S$ satisfies \eqref{eq:lie  conformalcom2} and \eqref{eq:lie conformal33} if and only if  $N+S$ is a Nijenhuis operator on
the semidirect product Lie conformal algebra $A\ltimes_{\rho}
V.$
\end{pro}

 \begin{defi}
Let $(A,[\cdot_\lambda\cdot];\rho)$ be a \LP pair. Let $N:A\rightarrow A$ and $S:V\rightarrow V$ be $\Comp[\partial]$-module homomorphisms. A pair $(N,S)$ is called a {\bf Nijenhuis structure} on the \LP pair $(A;\rho)$ if $N$ is a Nijenhuis operator on the Lie conformal algebra $A$ and for all $a\in A, v\in V$, the following conditions hold:
\begin{equation}\label{eq:lie conformal36}
  \rho(N(a))_{\lambda}S(v)=S(\rho(N(a))_{\lambda}(v))+\rho(a)_{\lambda}S^{2}(v)-S((\rho(a))_{\lambda}S(v)).
\end{equation}
 \end{defi}

Let $S:V\rightarrow V$ be an $\Comp[\partial]$-module homomorphism. Define
\begin{equation}
  S^*(\alpha)_\lambda v=\alpha_\lambda S(v),\quad\forall~\alpha\in {V^*}^c, v\in V.
\end{equation}
Since
\begin{eqnarray*}
S^*(\alpha)_\lambda (\partial v)=\alpha_\lambda S( \partial v)=\alpha_\lambda \partial S( v)=\lambda S^*(\alpha)_\lambda v,
\end{eqnarray*}
 we have $S^*(\alpha)\in {V^*}^c$.

\begin{pro}
Let $(A,[\cdot_\lambda\cdot];\rho)$ be an \LP pair. If $(N,S)$ is a Nijenhuis structure on the \LP pair $(A;\rho)$, then $N$ and $S^*$ generate a trivial infinitesimal deformation of the \LP pair $(A,[\cdot_\lambda\cdot];\rho^*)$.
\end{pro}
\begin{proof}
By a direct calculation, we have
\begin{eqnarray*}
S^*(\partial \alpha)_\lambda v=(\partial\alpha)_\lambda S(v)=-\lambda\alpha_\lambda S(v)=(\partial S^*( \alpha))_\lambda v,
\end{eqnarray*}
which implies that $S^*(\partial \alpha)=\partial S^*( \alpha)$.

Furthermore, by \eqref{eq:lie conformal36}, we have
\begin{eqnarray*}
&&(\rho^*(N(a))_{\lambda}(S^*(\alpha)))_\mu v-S^*(\rho^*(N(a))_{\lambda}(\alpha))_\mu v-S^*(\rho(a)_{\lambda}S^*(\alpha))_\mu v+{S^*}^{2}(\rho^*(a)_{\lambda}(\alpha))_\mu v\\
&=&-\alpha_{\mu-\lambda}(S\rho(N(a))_\lambda v)+ \alpha_{\mu-\lambda}(\rho(N(a))_{\lambda}S(v))+ \alpha_{\mu-\lambda}(S(\rho(a))_{\lambda}S(v))-\alpha_{\mu-\lambda}(\rho(a)_{\lambda}(S^{2}(v)))=0.
\end{eqnarray*}
By Theorem \ref{thm:trivial deform}, $N$ and $S^*$ generate a trivial infinitesimal deformation of the \LP pair $(A,[\cdot_\lambda\cdot];\rho^*)$
\end{proof}

\begin{cor}\label{cor:Hirac-Nijenhuis copair}
   Let $(N,S)$ be a Nijenhuis structure on a \LP pair $(A,[\cdot_\lambda\cdot];\rho)$. Then for any $i\in\Nat$, $(N^i,S^i)$ is a Nijenhuis structure on the \LP pair $(A;\rho)$.
\end{cor}
\begin{proof}
  It follows by a direct calculation.
\end{proof}

  \begin{ex}\label{ex:cNP}
 Let $N$ be a Nijenhuis operator on a Lie conformal algebra $(A,[\cdot_\lambda\cdot])$. Then $(N,N^*)$ is a Nijenhuis structure on the \LP pair $(A,[\cdot_\lambda\cdot];\ad^*)$.
 \end{ex}

Given a Nijenhuis structure $(N,S)$ on an \LP pair $(A,[\cdot_\lambda\cdot];\rho)$, we consider the associated deformed Lie conformal algebra $(A,[\cdot_\lambda\cdot]_N)$, and  construct a module over  $(A,[\cdot_\lambda\cdot]_N)$ which will be used in the following sections.

Let $N$ and $S$ be $\Comp[\partial]$-module homomorphisms. Define $\tilde{\rho}:A\longrightarrow{\rm gc}(V)$ by
 \begin{eqnarray}
    \tilde{\rho}(a)_{\lambda}(v)=\rho(Na)_{\lambda}(v)-\rho(a)_{\lambda}S(v)+S(\rho(a)_{\lambda}(v)),\quad\forall~a\in A,v\in V.
 \end{eqnarray}

\begin{pro}
If $(N,S)$ is a Nijenhuis structure on an \LP pair $(A;\rho)$, then $(V;\tilde{\rho})$ is a module over the Lie conformal algebra $(A,[\cdot_\lambda\cdot]_N)$.
\end{pro}
\begin{proof}
Note that for $a\in A,v\in V,\alpha\in V^{*c}$, we have
$$ (\tilde{\rho}^*(a)_{\lambda}(\alpha))_\mu v=(\rho^*(Na)_\lambda \alpha)_\mu v-(S^*\rho^*(a)_\lambda \alpha)_\mu v+(\rho^*(a)_\lambda S^*(\alpha))_\mu v,$$
which implies that
$$\tilde{\rho}^*(a)_{\lambda}(\alpha)=\rho^*(Na)_\lambda \alpha-S^*\rho^*(a)_\lambda \alpha+\rho^*(a)_\lambda S^*(\alpha).$$
Since $N+S^*$ is a Nijenhuis operator on the semidirect product Lie conformal algebra  $A\ltimes_{\rho^*}V^{*c}$, for $a,b\in A$ and $\alpha,\beta\in V^{*c}$, the deformed bracket $[\cdot_{\lambda}\cdot]_{N+S^*} $is given by
\begin{eqnarray*}
&&[(a+\alpha)_{\lambda}(b+\beta)]_{N+S^*}\\
&=&[(N+S^*)(a+\alpha)_{\lambda}(b+\beta)]_{\rho}+[(a+\alpha)_{\lambda}(N+S^*)(b+\beta)]_{\rho}-(N+S^*)[(a+\alpha)_{\lambda}(b+\beta)]_{\rho}\\
      &=&[N(a)_{\lambda}b]+[a_{\lambda}N(b)]-N[a_{\lambda}b]+\rho^*(Na)_{\lambda}(\beta)
      -\rho^*(b)_{-\lambda-\partial}S^*(\alpha)+\rho^*(a)_{\lambda}S^*(\beta)\\
      &&-\rho^*(Nb)_{-\lambda-\partial}(\alpha)-S^*\rho^*(a)_{\lambda}(\beta)+S^*\rho^*(b)_{-\lambda-\partial}(\alpha)\\
      &=&[a_{\lambda}b]_{N}+\tilde{\rho}^*(a)_{\lambda}(\beta)-\tilde{\rho}^*(b)_{-\lambda-\partial}(\alpha),
\end{eqnarray*}
which implies that $(V^{*c};\tilde{\rho}^*)$ is a module over the Lie conformal algebra $(A,[\cdot_\lambda\cdot]_N)$. Since the module $V$  is finite, $(V;\tilde{\rho})$ is a module over the Lie conformal algebra $(A,[\cdot_\lambda\cdot]_N)$.
\end{proof}

\section{$\GRBN$-structures and compatible $\huaO$-operators on  \LP pairs }\label{sec:ON-structures}
In this section, we give the notion of an $\GRBN$-structure on an \LP pair, which is a conformal analogue of the Possion-Nijenhuis structure.

Recall that a {\bf left-symmetric conformal algebra} $A$ is a $\Comp[\partial]$-module endowed with a $\Comp$-bilinear map $A\times A\rightarrow A[\lambda]$, denoted by $a\times b\rightarrow a\ast_{\lambda}b$, satisfying
\begin{eqnarray}
\label{eq:LSY conformal1} \partial a\ast_{\lambda}b&=&-\lambda a\ast_{\lambda}b,\\
\label{eq:LSY conformal2}  {a\ast_{\lambda} \partial b}&=&(\partial+\lambda)a\ast_{\lambda}b,\\
 \label{eq:LSY conformal3} (a\ast_{\lambda}b)\ast_{\lambda+\mu}c-a\ast_{\lambda}(b\ast_{\mu}c)&=&(b\ast_{\mu}a)\ast_{\lambda+\mu}c-b\ast_{\mu}(a\ast_{\lambda}c),\quad\forall~ a,b,c\in A.
\end{eqnarray}

Let $(A,\ast_\lambda)$ be a left-symmetric conformal algebra.  Then the $\lambda$-bracket
\begin{equation}
  [a_{\lambda}b]=a\ast_{\lambda}b-b\ast_{-\lambda-\partial}a \quad \forall a,b\in A
\end{equation}
defines a Lie conformal algebra, which is called {\bf the sub-adjacent Lie conformal algebra of} $A$. See \cite{Bakalov,HL} for more details.

\begin{lem}{\rm(\cite{HB})}
Let  $T:V\longrightarrow A$ be an $\GRB$-operator on an \LP pair  $(A,[\cdot_\lambda\cdot];\rho)$. Define a multiplication $\ast^T_\lambda$ on $V$ by
\begin{equation}
  u\ast_\lambda^T v=\rho(Tu)_\lambda(v),\quad \forall u,v\in V.
\end{equation}
Then $(V,\ast^T_\lambda)$ is a left-symmetric conformal algebra.
\end{lem}

We denote by $(V,[\cdot_{\lambda}\cdot]^{T})$ the sub-adjacent Lie conformal algebra of the left-symmetric conformal algebra $(V,\ast_{\lambda}^{T})$, where $[\cdot_\lambda \cdot]^T$ is given by
\begin{equation}\label{eq:lie conformal 45}
  [u_{\lambda}v]^{T}=\rho(T(u))_{\lambda}(v)-\rho(T(v))_{-\lambda-\partial}(u),\quad \forall~u,v\in V.
\end{equation}
Furthermore, $T$ is a Lie conformal algebra homomorphism from $(V,[\cdot_{\lambda}\cdot]^{T})$ to $(A,[\cdot_{\lambda}\cdot])$.

Let $T:V \rightarrow A$ be an  $\huaO$-operator and $(N,S)$ a Nijenhuis structure on an \LP pair $(A,[\cdot_{\lambda}\cdot];\rho)$. We define the bracket $[\cdot_{\lambda}\cdot]_{S}^{T}: \otimes^{2}V\rightarrow V[\lambda]$ to be the deformed bracket of $[\cdot_{\lambda}\cdot]^{T}$ by $S$, i.e.
\begin{equation}\label{eq:deformed S-bracket}
[u_{\lambda}v]_{S}^{T}=[S(u)_{\lambda}v]^{T}+[u_{\lambda}S(v)]^{T}-S([u_{\lambda}v]^{T}), \quad \forall~ u,v\in V.
\end{equation}
Define the bracket $\{\cdot_{\lambda}\cdot\}_{\tilde{\rho}}^{T}:\otimes^2 V\rightarrow V[\lambda]$ similar to \eqref {eq:lie conformal 45} using the module  $(V;\tilde{\rho})$,
\begin{equation}
\{u_{\lambda}v\}_{\tilde{\rho}}^{T}=\tilde{\rho}(T(u))_{\lambda}(v)-\tilde{\rho}(T(v))_{-\lambda-\partial}(u), \quad \forall~ u,v\in V.
\end{equation}
\begin{defi}
Let $T:V \rightarrow A$ be an  $\huaO$-operator and $(N,S)$ a Nijenhuis structure on an \LP pair $(A,[\cdot_{\lambda}\cdot];\rho)$. The triple $(T,N, S)$ is called an {\bf $\GRBN$-structure} on  the \LP pair $(A,[\cdot_{\lambda}\cdot];\rho)$ if $T$ and $(N,S)$ satisfy the following conditions
  \begin{eqnarray}
\label{eq:TS commute}     {N\circ T}&=& {T\circ S},\\
\label{eq:lie49}    {[u_{\lambda}v]^{N\circ T}}&=&{[u_{\lambda}v]_{S}^{T}},\quad \forall~ u,v\in V.
   \end{eqnarray}
\end{defi}

\begin{lem}\label{lem:3brackets}
 Let $(T,S,N)$ be an $\GRBN$-structure on  an \LP pair $(A,[\cdot_{\lambda}\cdot];\rho)$. Then we have
\begin{equation}
 {[u_{\lambda}v]^{N\circ T}}=[u_{\lambda}v]_{S}^{T}=\{u_{\lambda}v\}_{\tilde{\rho}}^{T},\quad \forall~ u,v\in V.
\end{equation}
\end{lem}
\begin{proof}
By \eqref{eq:TS commute}, we have $[u_{\lambda}v]_{S}^{T}+\{u_{\lambda}v\}_{\tilde{\rho}}^{T}=2[u_{\lambda}v]^{N\circ T}$. Furthermore, by \eqref{eq:lie49}, we have
$[u_{\lambda}v]_{S}^{T}=\{u_{\lambda}v\}_{\tilde{\rho}}^{T}$. The conclusion follows.
\end{proof}

\begin{pro}
Let $(T,S,N)$ be an $\GRBN$-structure on  an \LP pair $(A,[\cdot_{\lambda}\cdot];\rho)$.  Then $S$ is a Nijenhuis operator on the sub-adjacent Lie conformal algebra $(V,[\cdot_{\lambda}\cdot]^{T})$. Thus, the brackets  $[\cdot_{\lambda}\cdot]^T_S$, $[\cdot_{\lambda}\cdot]_{\tilde{\varrho}}^T$ and $[\cdot_{\lambda}\cdot]^{N\circ T }$ are all Lie conformal brackets.
\end{pro}
\begin{proof}
By the relation $[u_{\lambda}v]^{T\circ S}=[u_{\lambda}v]_{S}^{T}$, we have
\begin{eqnarray}
\label{eq:S-relation1}  \rho(T S(u))_{\lambda}S(v)-\rho(Tv)_{-\lambda-\partial}(S^{2}(u)) &=&S(\rho(T S(u))_{\lambda}(v))-S\rho(T(v))_{-\lambda-\partial}(S(u)),\\
\label{eq:S-relation2} S^{2}\rho(T(u))_{\lambda}(v)-S^{2}\rho(T(v))_{-\lambda-\partial}(u)&=& S\rho(Tu)_{\lambda}(S(v))-S\rho(T(v))_{-\lambda-\partial}S(u).
\end{eqnarray}
By \eqref{eq:lie conformal36}, we have
\begin{equation*}
  \rho(N T(v))_{\lambda}S(u)=S\rho(N T(v))_{\lambda}(u)+\rho(Tv)_{\lambda}S^{2}(u)-S\rho(Tv)_{\lambda}(S(u)),
\end{equation*}
which implies that
\begin{equation}
\label{eq:S-relation3}  \rho(T(v))_{\lambda}(S^{2}(u))-\rho(T S(v))_{\lambda}(S(u))=S\rho(T(v))_{\lambda}(S(u))-S\rho(T\circ S(v))_{\lambda}(u).
\end{equation}
By \eqref{eq:S-relation1}-\eqref{eq:S-relation3}, we have
\begin{equation*}
  \begin{split}
    &[S(u)_{\lambda}S(v)]^{T}-S([u_{\lambda}v]_{S}^{T})\\
   =&\rho(T S(u))_{\lambda}(S(v))-\rho(T S(v))_{-\lambda-\partial}(S(u))-S\rho(T S(u))_{\lambda}(v)+S\rho(T(v))_{-\lambda-\partial}(S(u))\\
    &-S\rho(T(u))_{\lambda}S(v)+S\rho(T S(v))_{-\lambda-\partial}(u)+S^{2}\rho(T(u))_{\lambda}(v)-S^{2}\rho(T(v))_{-\lambda-\partial}(u)\\
   =&\rho(T S(u))_{\lambda}(S(v))-\rho(T S(v))_{-\lambda-\partial}(S(u))-\rho(T S(u))_{\lambda}(S(v))+\rho(T(v))_{-\lambda-\partial}(S^{2}(u))\\
    &-S\rho(T(u))_{\lambda}(S(v))+S \rho(T S(v))_{-\lambda-\partial}(u)+S\rho(T(u))_{\lambda}(S(v))-S\rho(T(v))_{-\lambda-\partial}S(u)\\
   =&\rho(T(v))_{-\lambda-\partial}(S^{2}(u))-\rho(T S(v))_{-\lambda-\partial}(S(u))-S\rho(T(v))_{-\lambda-\partial}(S(u))+S\rho(T S(v))_{-\lambda-\partial}(u)\\
    =& 0.
  \end{split}
\end{equation*}
Thus $S$ is a Nijenhuis operator on the Lie conformal algebra $(V,[\cdot_{\lambda}\cdot]^{T})$.  By Corollary \ref{cor:Nijenhuis}, $[\cdot_{\lambda}\cdot]^{T}_S$ is a Lie conformal bracket on $V$. By Lemma \ref{lem:3brackets}, the brackets  $[\cdot_{\lambda}\cdot]^T_S$, $[\cdot_{\lambda}\cdot]_{\tilde{\varrho}}^T$ and $[\cdot_{\lambda}\cdot]^{N\circ T }$ are all Lie conformal brackets.
\end{proof}

\begin{thm}\label{thm:lie5.8}
Let $(T,S,N)$ be an $\GRBN$-structure on  an \LP pair $(A,[\cdot_{\lambda}\cdot];\rho)$. Then we have
\begin{itemize}
\item[$\rm(i)$] $T$ is an $\huaO$-operator on the \LP pair $(A,[\cdot_{\lambda}\cdot]_{N};\tilde{\rho})$.
\item[$\rm(ii)$] $N\circ T$ is an $\huaO$-operator on the \LP pair $(A,[\cdot_{\lambda}\cdot];{\rho})$.
  \end{itemize}
  \end{thm}
\begin{proof}
(i) Since $T$ is an $\huaO$-operator on the \LP pair $(A,[\cdot_{\lambda}\cdot];\rho)$
and $T\circ S=N\circ T$, we have
\begin{equation*}
  \begin{split}
    T\{u_{\lambda}v\}_{\tilde{\rho}}^{T}&=T([u_{\lambda}v]_{S}^{T})=T([S(u)_{\lambda}(v)]^{T}+[u_{\lambda}S(v)]^{T}-S[u_{\lambda}v]^{T})\\
      &=[T S(u)_{\lambda}T(v)]+[T(u)_{\lambda}T S(v)]-T S[u_{\lambda}v]^{T}\\
      &=[N T(u)_{\lambda}T(v)]+[T(u)_{\lambda}N T(v)]-N T[u_{\lambda}v]^{T}\\
      &=[T(u)_{\lambda}T(v)]_{N},
  \end{split}
\end{equation*}
which implies that $T$ is an $\huaO$-operator on the \LP pair $(A,[\cdot_{\lambda}\cdot]_{N};\tilde{\rho})$.

(ii) By \eqref{eq:lie49}, we have
$$ N T([u_{\lambda}v]^{N\circ T})=N T([u_{\lambda}v]_{S}^{T})=N[T(u)_{\lambda}T(v)]_{N}=[N T(u)_{\lambda}N T(v)],$$
which implies that $N\circ T$ is an $\huaO$-operator on the \LP pair $(A,[\cdot_{\lambda}\cdot];{\rho})$.
\end{proof}

In the following, we give the notion of compatible $\huaO$-operators on an \LP pair, and show that an $\GRBN$-structure gives rise to compatible $\huaO$-operators.

\begin{defi}
Let $T_{1},T_{2}$:$V\rightarrow A$ be two $\huaO$-operators on an \LP pair $(A,[\cdot_{\lambda}\cdot];\rho)$. If for any $k_{1},k_{2}\in\Real$, $k_{1}T_{1}+k_{2}T_{2}$ is still an $\huaO$-operator, then $T_{1}$ and $T_{2}$ are called {\bf compatible}.
\end{defi}

\begin{pro}
Let $T_{1},T_{2}$:$V\rightarrow A$ be two $\huaO$-operators on an \LP pair $(A,[\cdot_{\lambda}\cdot];\rho)$. Then $T_{1}$ and $T_{2}$ are compatible if and only if the following equation holds:
\begin{equation}\label{eq:lie56}
  \begin{split}
     [T_{1}(u)_{\lambda}T_{2}(v)]+[T_{2}(u)_{\lambda}T_{1}(v)]& =T_{1}(\rho(T_{2}(u))_{\lambda}(v)-\rho(T_{2}(v))_{-\lambda-\partial}(u))\\
      &+T_{2}(\rho(T_{1}(u))_{\lambda}(v)-\rho(T_{1}(v))_{-\lambda-\partial}(u)), \quad \forall u,v\in V.
   \end{split}
\end{equation}
\end{pro}
\begin{proof}
It follows by a direct calculation.
\end{proof}

\begin{ex}
  Let $A=\Comp[\partial]a\oplus \Comp[\partial]b$ be the quadratic Lie conformal algebra given in Example \ref{ex:N-operator rank2}. Define $R_1,R_2:A\longrightarrow A$ by
  \begin{alignat*}{2}
    R_1(a)&=-(a+b),\quad& R_1(b)&=(a+b),\\
     R_2(a)&=(a+b),\quad&  R_2(b)&=-(a+b).
  \end{alignat*}
  Then $R_1$ and $R_2$ are compatible Rota-Baxter operators.
\end{ex}

Using an  $\huaO$-operator and a Nijenhuis operator on a Lie conformal algebra, we can construct a pair of compatible $\huaO$-operators.

\begin{pro}\label{pro:lie6.3}
Let $T:V\rightarrow A$ be an $\huaO$-operator on an \LP pair $(A,[\cdot_{\lambda}\cdot];\rho)$ and $N$ be a Nijenhuis operator on $A$. Then $N\circ T$ is an $\huaO$-operator on $(A,[\cdot_{\lambda}\cdot];\rho)$ if and only if for all $u,v\in V$, the following equation holds:
\begin{equation}\label{eq:lie57}
  \begin{split}
    &N([NT(u)_{\lambda}T(v)]+[T(u)_{\lambda}NT(v)]) \\
      =&N(T(\rho(NT(u))_{\lambda}(v)-\rho(NT(v))_{-\lambda-\partial}(u))+NT(\rho(T(u))_{\lambda}(v)-\rho(T(v))_{-\lambda-\partial}(u))).
  \end{split}
\end{equation}
If, in addition, $N$ is invertible, then $T$ and $N\circ T$ are compatible. More explicitly, for any $\huaO$-operator $T$, if there exists an invertible Nijenhuis operator $N$ such that $N \circ T$ is also an  $\huaO$-operator, then $T$ and $N\circ T$ are compatible.
\end{pro}
\begin{proof}
Since $N$ is a Nijenhuis operator, we have
$$ [NT(u)_{\lambda}NT(v)]=N([NT(u)_{\lambda}T(v)]+[T(u)_{\lambda}NT(v)])-N^{2}([T(u)_{\lambda}T(v)]).$$
By the fact that $T$ is an $\huaO$-operator, the following equation
$$[NT(u)_{\lambda}NT(v)]=NT\big(\rho(NT(u))_{\lambda}(v)-\rho(NT(v))_{-\lambda-\partial}(u)\big)$$
holds if and only if \eqref{eq:lie57} holds.

If $N\circ T$ is an $\huaO$-operator and $N$ is invertible, then we have
\begin{eqnarray*}
[NT(u)_{\lambda}T(v)]+[T(u)_{\lambda}NT(v)]&=&T(\rho(NT(u))_{\lambda}(v)-\rho(NT(v))_{-\lambda-\partial}(u))\\
&&+NT(\rho(T(u))_{\lambda}(v)-\rho(T(v))_{-\lambda-\partial}(u)).
\end{eqnarray*}
This is exactly the condition that $N\circ T$ and $T$ are compatible.
\end{proof}

A pair of compatible $\huaO$-operators can also give rise to a Nijenhuis operator under an invertible condition.
\begin{pro}\label{pro:lie6.4}
Let $T_{1},T_{2}:V\rightarrow A$ be two $\huaO$-operators on an \LP pair $(A,[\cdot_{\lambda}\cdot];\rho)$. Suppose that $T_{2}$ is invertible. If $T_{1}$, $T_{2}$ are compatible, then $N=T_{1}\circ T_{2}^{-1}$ is a Nijenhuis operator on the Lie conformal algebra $A$.
\end{pro}
\begin{proof}
Since $T_{1}=N\circ T_{2}$ is an $\huaO$-operator, we have
$$  [NT_{2}(u)_{\lambda}NT_{2}(v)]=NT_{2}(\rho(NT_{2}(u)_{\lambda}(v)-\rho(NT_{2}(v)_{-\lambda-\partial}(u)).$$
Since $T_{1}$ and $T_{2}$ are compatible, we have
\begin{equation*}
  \begin{split}
    &[NT_{2}(u)_{\lambda}T_{2}(v)]+[T_{2}(u)_{\lambda}NT_{2}(v)] \\
      =&T_{2}(\rho(NT_{2}(u))_{\lambda}(v)-\rho(NT_{2}(v))_{-\lambda-\partial}(u))+NT_{2}(\rho(T_{2}(u)){\lambda}(v)-\rho(T_{2}(v))_{-\lambda-\partial}(u))\\
      =&T_{2}(\rho(NT_{2}(u))_{\lambda}(v)-\rho(NT_{2}(v))_{-\lambda-\partial}(u))+N([T_{2}(u)_{\lambda}T_{2}(v)]).
  \end{split}
\end{equation*}
Thus we have
\begin{eqnarray*}
  &&[NT_{2}(u)_{\lambda}NT_{2}(v)]-N[NT_{2}(u)_{\lambda}T_{2}(v)]-N[T_{2}(u)_{\lambda}NT_{2}(v)]+N^{2}([T_{2}(u)_{\lambda}T_{2}(v)])\\
  &=&NT_{2}(\rho(NT_{2}(u)_{\lambda}(v)-\rho(NT_{2}(v)_{-\lambda-\partial}(u))+T_{2}(\rho(NT_{2}(u))_{\lambda}(v)+\rho(NT_{2}(v))_{-\lambda-\partial}(u))\\
  &&-N^2([T_{2}(u)_{\lambda}T_{2}(v)])+N^2([T_{2}(u)_{\lambda}T_{2}(v)])\\
  &=&0.
\end{eqnarray*}
Since $T_{2}$ is invertible, for any $a,b\in A$, there exist $u,v\in V$ such that $T_{2}(u)=a$ and $T_{2}(v)=b$.  The above equality implies that
$$[Na_{\lambda}Nb]=N[Na_{\lambda}b]+N[a_{\lambda}Nb]-N^{2}[a_{\lambda}b].$$
Thus $N$ is a Nijenhuis operator on the Lie conformal algebra $A$.
\end{proof}
By Propositions \ref{pro:lie6.3} and \ref{pro:lie6.4}, we have

\begin{cor}
Let $T_{1}, T_{2}:V\rightarrow A$ be two $\huaO$-operators on an \LP pair $(A,[\cdot_{\lambda}\cdot];\rho)$. Suppose $T_{1}$ and $T_{2}$ are invertible. Then $T_{1}$ and $T_{2}$ are compatible if and only if $N=T_{1}\circ T_{2}^{-1}$ is a Nijenhuis operator on $A$.
\end{cor}

The following proposition shows that compatible $\huaO$-operators can be obtained from $\GRBN$-structures.

\begin{pro}\label{pro:compatible O}
Let $(T,S,N)$ be an $\GRBN$-structure on an \LP pair $(A,[\cdot_{\lambda}\cdot];\rho)$. Then $T$ and $T\circ S$ are compatible $\huaO$-operators.
\end{pro}
\begin{proof}
It is sufficient to prove that $T+T\circ S$ is an $\huaO$-operator. It is not hard to see that
$$ [u_{\lambda}v]^{T+T\circ S}=[u_{\lambda}v]^{T}+[u_{\lambda}v]^{T\circ S}=[u_{\lambda}v]^{T}+[u_{\lambda}v]_{S}^{T}.$$
Thus we have
\begin{equation*}
  \begin{split}
    &(T+T\circ S)([u_{\lambda}v]^{T+T\circ S})\\
      =& T([u_{\lambda}v]^{T})+T\circ S([u_{\lambda}v]_{S}^{T})+T\circ S([u_{\lambda}v]^{T})+T([u_{\lambda}v]_{S}^{T})\\
      =& T([u_{\lambda}v]^{T})+T\circ S([u_{\lambda}v]_{S}^{T})+T\circ S([u_{\lambda}v]^{T})\\
      &+T([S(u)_{\lambda}v]^{T}+[u_{\lambda}S(v)]^{T}-S[u_{\lambda}v]^{T})\\
      =& T([u_{\lambda}v]^{T})+T\circ S([u_{\lambda}v]_{S}^{T})+T([S(u)_{\lambda}v]^{T}+[u_{\lambda}S(v)]^{T})\\
      =& [T(u)_{\lambda}T(v)]+[T\circ S(u)_{\lambda}T\circ S(v)]+[T\circ S(u)_{\lambda}T(v)]+[T(u)_{\lambda}T\circ S(v)]\\
      =& [(T+T\circ S)(u)_{\lambda}(T+T\circ S)(v)],
  \end{split}
\end{equation*}
which implies that $T+T\circ S$ is an $\huaO$-operator.
\end{proof}

Similar to the properties of $\GRBN$-structures on Lie algebras given in \cite{HLS}, we also have
\begin{lem}
Let $(T,S,N)$ be an $\GRBN$-structure on an \LP pair $(A,[\cdot_{\lambda}\cdot];\rho)$. Then
\begin{itemize}
\item[$\rm(i)$] for all $k,i\in \Nat$, we have
\begin{equation} \label{eq:lie58}
  T_{k}[u_{\lambda}v]_{S^{k+i}}^{T}=[T_{k}(u)_{\lambda}T_{k}(v)]_{N^{i}};
\end{equation}
\item[$\rm(ii)$] for all $k,i\in \Nat$ such that $i\leq k$,
\begin{equation}\label{eq:lie61}
  [u_{\lambda}v]^{T^{k}}=[u_{\lambda}v]_{S^{k}}^{T}=S^{k-i}[u_{\lambda}v]^{T_{i}},
\end{equation}
where $T_{k}=T\circ S^{k}=N^{k}\circ T$ and set $T_{0}=T$.
  \end{itemize}
\end{lem}
\begin{proof}
  The proof  is similar to that of Lemmas 4.8 and 4.9 in \cite{HLS}.
\end{proof}

\begin{pro}\label{pro:compatible O2}
Let $(T,S,N)$ be an $\GRBN$-structure on an \LP pair $(A,[\cdot_{\lambda}\cdot];\rho)$.  Then all $T_{k}=N^{k}\circ T$ are $\huaO$-operators on the \LP pair $(A;\rho)$ and for all $k,l\in \Nat$, $T_{k}$ and $T_{l}$ are compatible.
\end{pro}
\begin{proof}
By \eqref{eq:lie58} and \eqref{eq:lie61} with $i=0$, we have
$$ T_{k}[u_{\lambda}v]^{T_{k}}=[T_{k}(u)_{\lambda}T_{k}(v)],$$
which implies that $T_{k}$ is an $\huaO$-operator.

For the second assertion, we only need to prove that $T^{k}+T^{k+i}$ is an $\huaO$-operator for all $k,i\in \Nat$. By \eqref{eq:lie61}, we have
$$ [u_{\lambda}v]^{T_{k}+T_{k+i}}=[u_{\lambda}v]^{T_{k}}+[u_{\lambda}v]^{T_{k+i}}=[u_{\lambda}v]^{T_{k}}+[u_{\lambda}v]_{S^{i}}^{T_{k}}.$$
Thus, we have
\begin{equation*}
  \begin{split}
    &(T_{k}+T_{k+i})([u_{\lambda}v]^{T_{k}+T_{k+i}})\\
  =& T_{k}([u_{\lambda}v]^{T_{k}})+T_{k}([u_{\lambda}v]_{S^{i}}^{T_{k}})+T_{k+i}([u_{\lambda}v]^{T_{k}})+T_{k+i}([u_{\lambda}v]_{S^{i}}^{T_{k}})\\
  =& T_{k}([u_{\lambda}v]^{T_{k}})+T_{k+i}([u_{\lambda}v]^{T_{k}})+T_{k+i}([u_{\lambda}v]_{S^{i}}^{T_{k}})\\
  &+T_{k}([S^{i}(u)_{\lambda}v]^{T_{k}}+[u_{\lambda}S^{i}(v)]^{T_{k}}-S^{i}[u_{\lambda}v]^{T_{k}})\\
  =&T_{k}([u_{\lambda}v]^{T_{k}})+T_{k+i}([u_{\lambda}v]_{S^{i}}^{T_{k}})+T_{k}([S^{i}(u)_{\lambda}v]^{T_{k}})+T_{k}([u_{\lambda}S^{i}(v)]^{T_{k}})\\
  =& [T_{k}(u)_{\lambda}T_{k}(v)]+[T_{k+i}(u)_{\lambda}T_{k+i}(v)]+[T_{k+i}(u)_{\lambda}T_{k}(v)]+[T_{k}(u)_{\lambda}T_{k+i}(v)]\\
  =&[(T_{k}+T_{k+i})(u)_{\lambda}(T_{k}+T_{k+i})(v)],
  \end{split}
\end{equation*}
which implies that $T_{k}+T_{k+i}$ is an $\huaO$-operator. Thus $T_{k}$ and $T_{l}$ are compatible for all $k,l\in\Nat$.
\end{proof}

Compatible $\huaO$-operators can give rise to $\GRBN$-structures.

\begin{thm}\label{thm:ON construction}
Let $T,T_{1}:V\rightarrow A$ be two $\huaO$-operators on an \LP pair $(A,[\cdot_{\lambda}\cdot];\rho)$. Suppose that $T$ is invertible. If $T$ and $T_{1}$ are compatible, then
\begin{itemize}
\item[$\rm(i)$]$(T,S=T^{-1}\circ T_1,N=T_1\circ T^{-1})$ is an $\GRBN$-structure on the \LP pair $(A,[\cdot_{\lambda}\cdot];\rho)$;
\item[$\rm(ii)$]$(T_1,S=T^{-1}\circ T_1,N=T_1\circ T^{-1})$ is an $\GRBN$-structure on the \LP pair $(A,[\cdot_{\lambda}\cdot];\rho)$.
\end{itemize}
\end{thm}
\begin{proof}
By Proposition \ref{pro:lie6.4}, $N=T_{1}\circ T^{-1}$ is a Nijenhuis operator on $A$.

Since $T$ and $T_{1}$ are compatible, we have
\begin{eqnarray*} [T(u)_{\lambda}T_{1}(v)]+[T_{1}(u)_{\lambda}T(v)]&=&T(\rho(T_{1}(u))_{\lambda}(v)-\rho(T_{1}(v))_{-\lambda-\partial}(u))\\
&&+T_{1}(\rho(T(u))_{\lambda}(v)-\rho(T(v))_{-\lambda-\partial}(u)).
\end{eqnarray*}
Because of $T_{1}=T\circ S$,  we get
\begin{equation}\label{eq:lie62}
\begin{split}
  [T(u)_{\lambda}TS(v)]+[T S(u)_{\lambda}T(v)]=& T(\rho(T S(u))_{\lambda}(v)-\rho(T S(v))_{-\lambda-\partial}(u)) \\
    &+ T S(\rho(T(u))_{\lambda}(v)-\rho(T(v))_{-\lambda-\partial}(u)).
\end{split}
\end{equation}
Since $T$ is an $\huaO$-operator, we have
\begin{equation}\label{eq:lie63}
  \begin{split}
    [T(u)_{\lambda}T S(v)]+[T S(u)_{\lambda}T(v)]= &T(\rho(T(u))_{\lambda}(S(v))-\rho(T S(v))_{-\lambda-\partial}(u) \\
      &+T(\rho(T S(u))_{\lambda}(v)-\rho(T(v))_{-\lambda-\partial}(S(u)).
  \end{split}
\end{equation}
By \eqref{eq:lie62}, \eqref{eq:lie63} and the fact that $T$ is invertible, we have
\begin{equation}\label{eq:lie64}
  S(\rho(T(u))_{\lambda}(v)-\rho(T(v))_{-\lambda-\partial}(u))=\rho(T(u))_{\lambda}(S(v))-\rho(T(v))_{-\lambda-\partial}(S(u)).
\end{equation}
Replacing $v$ with $S(v)$, we get
\begin{equation}\label{eq:lie65}
  S(\rho(T(u))_{\lambda}(S(v))-\rho(T\circ S(v))_{-\lambda-\partial}(u))=\rho(T(u))_{\lambda}(S^{2}(v))-\rho(T\circ S(v))_{-\lambda-\partial}(S(u)).
\end{equation}
Since $T$ and $T\circ S$ are $\huaO$-operators, we have
$$ T S[u_{\lambda}v]^{T\circ S}=[T S(u)_{\lambda}T S(v)]=T[S(u)_{\lambda}S(v)]^{T}.$$
By the fact that $T$ is invertible, we have
$$ S[u_{\lambda}v]^{T\circ S}=[S(u)_{\lambda}S(v)]^{T}.$$
which implies that
\begin{equation}\label{eq:LC66}
  S(\rho((T S)(u))_{\lambda}(v)-\rho(T S(v))_{-\lambda-\partial}(u))=\rho((T S)(u))_{\lambda}(S(v))-\rho((T S)(v))_{-\lambda-\partial}(S(u)).
\end{equation}
By \eqref{eq:lie65} and \eqref{eq:LC66}, we have
$$ S(\rho(T(u))_{\lambda}(S(v)))-\rho(T(u))_{\lambda}(S^{2}(v))-S(\rho((T\circ S)(u))_{\lambda}(v))+\rho((T\circ S)(u))_{\lambda}(S(v))=0.$$
By the fact that $N\circ T=(T_{1}\circ T^{-1})\circ T=T_{1}=T\circ S$, we have
$$ S(\rho(T(u))_{\lambda}(S(v)))-\rho(T(u))_{\lambda}(S^{2}(v))-S(\rho(N\circ T)(u))_{\lambda}(v))+\rho((N\circ T)(u))_{\lambda}(S(v))=0.$$
Since $T$ is invertible and let $a=T(u)$, we have
$$  S(\rho(a)_{\lambda}(S(v)))-\rho(a)_{\lambda}(S^{2}(v))-S(\rho(N(a))_{\lambda}(v))+\rho(N(a))_{\lambda}(S(v))=0.$$
Thus $(N,S)$ is a Nijenhuis structure on an \LP pair $(A;\rho)$.

(i) It is obvious that $N\circ T=T\circ S$. By \eqref{eq:lie64}, we have
\begin{equation*}
  \begin{split}
     [u_{\lambda}v]_{S}^{T}-[u_{\lambda}v]^{T\circ S}&=\rho(T(u))_{\lambda}(S(v))-\rho(T(v))_{-\lambda-\partial}(S(u))
     -S(\rho(T(u))_{\lambda}(v)-\rho(T(v))_{-\lambda-\partial}(u)) \\
       &=0,
   \end{split}
\end{equation*}
which implies that $[u_{\lambda}v]_{S}^{T}=[u_{\lambda}v]^{T\circ S}$. Therefore, $(T,S=T^{-1}\circ T_1,N=T_1\circ T^{-1})$ is an $\GRBN$-structure on the \LP pair $(A,[\cdot_{\lambda}\cdot];\rho)$.

(ii) It is obvious that $T_{1}\circ S=N\circ T_{1}$. By \eqref{eq:LC66}, we have
\begin{equation*}
  \begin{split}
    &[u_{\lambda}v]_{S}^{T_{1}}-[u_{\lambda}v]^{T_{1}\circ S} \\
  = &\rho(T_{1}(u))_{\lambda}(S(v))-\rho(T_{1}(v))_{-\lambda-\partial}(S(u))-S(\rho(T_{1}(u))_{\lambda}(v)-\rho(T_{1}(v))_{-\lambda-\partial}(u)) \\
  = &\rho((T\circ S)(u))_{\lambda}(S(v))-\rho((T\circ S)(v))_{-\lambda-\partial}(S(u))-S(\rho((T\circ S)(u))_{\lambda}(v)-\rho((T\circ S)(v))_{-\lambda-\partial}(u))\\
  =&0,
   \end{split}
\end{equation*}
which implies that $[u_{\lambda}v]_{S}^{T_{1}}=[u_{\lambda}v]^{T_{1}\circ S}$. Thus $(T_1,S=T^{-1}\circ T_1,N=T_1\circ T^{-1})$ is an $\GRBN$-structure on the \LP pair $(A,[\cdot_{\lambda}\cdot];\rho)$.
\end{proof}

\begin{pro}\label{pro:GRBN-invertible}
 Let $(T,N,S)$ be an $\GRBN$-structure on an \LP pair $(A;\rho)$. If $T$ is invertible, then $(T,N^k,S^k)$ and $(T_k=N^k\circ T,N^k,S^k)$ are $\GRBN$-structures for all $k\in\Nat$.
\end{pro}
\begin{proof}
Since $(T,N,S)$ is an $\GRBN$-structure on the \LP pair $(A;\rho)$, by Proposition \ref{pro:compatible O2},  $T$ and $T_k=N^k\circ T$ are compatible $\GRB$-operators.  Then by the condition that $T$ is invertible and Theorem \ref{thm:ON construction}, the conclusion follows.
\end{proof}

\section{Conformal $r$-matrix-Nijenhuis structures and symplectic-Nijenhuis structures on Lie conformal algebras}\label{sec:rn}
In this section, we introduce the notions of conformal $r$-matrix-Nijenhuis structure and symplectic-Nijenhuis structure on a Lie conformal algebra. Furthermore, we study their connections with $\GRBN$-structures on \LP pairs.

\subsection{Conformal $r$-matrix-Nijenhuis structures on Lie conformal algebras}
Let $V$ be free and of finite rank as a $\Comp[\partial]$-module. Then $V\cong {({V^*}^c)^*}^c$ through the $\Comp[\partial]$-module homomorphism $v\mapsto v^{**},$ where $v^{**}$ is given by
$$v^{**}_\lambda(\alpha)=\alpha_{-\lambda}(v),\quad\forall~ \alpha \in{V^*}^c .$$

Suppose $A$ is a finite Lie conformal algebra which is free as a $\Comp[\partial]$-module. Then we have
$$ A\otimes A\cong A^{\ast c \ast c}\otimes A \cong {\rm Chom}(A^{\ast c},A).$$
Define
$$ \langle\alpha,a\rangle_{\lambda}= \alpha_{\lambda}(a), \quad and \quad\langle\alpha\otimes \beta, a\otimes b\rangle_{(\lambda,\mu)}=\langle\alpha,a\rangle_{\lambda}\langle\beta,b\rangle_{\mu},$$
where $a,b\in A$ and $\alpha,\beta\in A^{* c}$. Then for any $r\in A\otimes A$ with $r=\sum_{i} a_i\otimes b_i\in A\otimes A$, we associate a conformal linear map $r^{\sharp}\in {\rm Chom}(A^{\ast c},A)$ as follows:
\begin{equation}
  r^{\sharp}_{\lambda}(\alpha)=\sum_i\langle\alpha,a_i\rangle_{-\lambda-\partial}b_i,\quad \forall~\alpha\in A^{\ast c},
\end{equation}
which is equivalent to
\begin{equation}\label{eq:lie66}
  \langle\beta,r^{\sharp}_{-\mu-\partial}(\alpha)\rangle_{\lambda}=\langle \alpha\otimes \beta,r\rangle_{(\mu,\lambda)}, \quad \forall~\alpha,\beta\in A^{\ast c}.
\end{equation}

Set $r^{21}=\sum_i b_i\otimes a_i$. We say $r$ is {\bf skew-symmetric} if $r=-r^{21}$.

We say that $r$ is {\bf non-degenerate} if $r_{\lambda}^{\sharp}=r^{\sharp}_{\lambda}|_{\lambda=0}$ and $r_{0}^{\sharp}$ is a $\Comp[\partial]$-module isomorphism from $A^{\ast c}$ to $A$. In this case, \eqref{eq:lie66} becomes
\begin{equation}\label{eq:non-degenerate r-matrix}
\langle\beta,r_{0}^{\sharp}(\alpha)\rangle_{\lambda}=\langle\alpha\otimes \beta,r\rangle_{(-\lambda,\lambda)} \quad\forall~ \alpha,\beta\in A^{\ast c}.
\end{equation}

\emptycomment{Set $r=\sum_{i} a_i\otimes b_i\in A\otimes A$. Then we have
\begin{equation}
  r^{\sharp}_{\lambda}(\alpha)=\sum_i\{\alpha,a_i\}_{-\lambda-\partial}b_i,\quad \forall~\alpha\in A^{\ast c}.
\end{equation}}

\begin{defi}{\rm(\cite{Lib})}
Let $A$ be a Lie conformal algebra. An element $r\in A\otimes A$ is called a {\bf conformal $r$-matrix } if $\llbracket r,r\rrbracket\equiv0~{\rm mod} ~\partial^{\otimes^{3}}$ in $A\otimes A\otimes A$, where $\partial^{\otimes^{3}}=\partial\otimes1\otimes1+1\otimes\partial\otimes1+1\otimes1\otimes\partial.$
\end{defi}

There is a close relationship between conformal $r$-matrices and $\huaO$-operators.

\begin{lem}{\rm(\cite{HB})}\label{lem:equivalence r-matrix}
Let $A$ be a finite Lie conformal algebra. Then, $r$ is a skew-symmetric conformal $r$-matrix if and only if $r^{\sharp}_{0}=r^{\sharp}_{\lambda}|_{\lambda=0}$ is an $\huaO$-operator on the \LP pair $(A;\ad^*)$.
\end{lem}

\begin{defi}
Let $r$ be a skew-symmetric conformal $r$-matrix and $N:A\rightarrow A$ a Nijenhuis operator on a finite Lie conformal algebra A. A pair $(r,N)$ is called a {\bf conformal $r$-matrix-Nijenhuis structure} on the Lie conformal algebra $A$  if for all $\alpha, \beta\in A^{\ast c}$, the following conditions are satisfied:
\begin{eqnarray}
 {N\circ r_\lambda^{\sharp} }&=&{r_\lambda^{\sharp}\circ N^{\ast}},\\
  {[\alpha_{\lambda}\beta]^{N\circ r_0^{\sharp}}} &=& {[\alpha_{\lambda}\beta]_{N^{\ast}}^{r_0^{\sharp}}}.
\end{eqnarray}
\end{defi}
There is a close relationship between conformal $r$-matrix-Nijenhuis structures and $\GRBN$-structures on Lie conformal algebras.

\begin{thm}\label{thm:rNijRB}
Let $r$ be a skew-symmetric conformal $r$-matrix and $N:A\rightarrow A$ a Nijenhuis operator on a Lie conformal algebra $A$. If $(r,N)$ is a conformal $r$-matrix-Nijenhuis structure on $A$, then $(r_0^{\sharp},N,S=N^{\ast})$ is an $\GRBN$-structure on the \LP pair $(A;\ad^{\ast})$.
\end{thm}
\begin{proof}
  It is straightforward.
\end{proof}

The following proposition shows that if the conformal $r$-matrix $r$ is determined by $r_{0}^{\sharp}$, the conformal $r$-matrix-Nijenhuis structure is equivalent to the $\GRBN$-structure on the \LP pair $(A;\ad^{\ast})$.
\begin{pro}\label{pro:non-degenerate}
  Let $r$ be a skew-symmetric conformal $r$-matrix with $r_{\lambda}^{\sharp}=r_{0}^{\sharp}$  and $N:A\rightarrow A$ a Nijenhuis operator on a Lie conformal algebra $A$. Then $(r,N)$ is a conformal $r$-matrix-Nijenhuis structure on $A$ if and only if $(r_0^{\sharp},N,S=N^{\ast})$ is an $\GRBN$-structure on the \LP pair $(A;\ad^{\ast})$.
\end{pro}
\begin{proof}
  Since  $r_{\lambda}^{\sharp}=r_{0}^{\sharp}$,  ${N\circ r_\lambda^{\sharp} }={r_\lambda^{\sharp}\circ N^{\ast}}$ holds if and only if ${N\circ r_0^{\sharp} }={r_0^{\sharp}\circ N^{\ast}}$ holds. The rest is straightforward.
\end{proof}

\begin{pro}\label{pro:compatible r-matrix}
 Assume that $(r,N)$ is skew-symmetric conformal $r$-matrix-Nijenhuis structure on $A$ and set $r=\sum_{i} a_i\otimes b_i$.   Then $r_N=\sum_{i} a_i\otimes N(b_i)$ is skew-symmetric. Furthermore,  $r$ and $r_{N}$ are compatible conformal $r$-matrices in the sense that any linear combination of $r$ and $r_{N}$ is still a conformal $r$-matrix.
\end{pro}
\begin{proof}
 By the fact that the Lie conformal algebra $A$ is finite,  $r$ is skew-symmetric if and only if
 $$\langle \beta\otimes \alpha,r\rangle_{(\mu,\lambda)}=-\langle \alpha\otimes \beta,r\rangle_{(\lambda,\mu)},\quad \forall~\alpha,\beta\in A^{*c},$$
which is also equivalent to
 $$\langle\alpha,r^{\sharp}_{-\mu-\partial}(\beta)\rangle_{\lambda}=-\langle\beta,r^{\sharp}_{-\lambda-\partial}(\alpha)\rangle_{\mu}.$$
Furthermore, by \eqref{eq:lie66} and  ${N\circ r_\lambda^{\sharp} }={r_\lambda^{\sharp}\circ N^{\ast}}$, we have
\begin{eqnarray*}
\langle\alpha\otimes \beta,r_N\rangle_{(\lambda,\mu)}&=& \langle\alpha\otimes N^*\beta,r\rangle_{(\lambda,\mu)} =\langle N^*\beta,r^{\sharp}_{-\mu-\partial}(\alpha)\rangle_{\lambda}\\
&=&-\langle\alpha,r^{\sharp}_{-\lambda-\partial}(N^*\beta)\rangle_{\mu}=-\langle\alpha,Nr^{\sharp}_{-\lambda-\partial}(\beta)\rangle_{\mu}\\
&=&-\langle N^*\alpha,r^{\sharp}_{-\lambda-\partial}(\beta)\rangle_{\mu}=-\langle\beta\otimes N^*\alpha,r\rangle_{(\lambda,\mu)}\\
&=&-\langle\beta\otimes \alpha,r_N\rangle_{(\mu,\lambda)},
\end{eqnarray*}
which implies that $r_N$ is skew-symmetric.

By Theorem \ref{thm:rNijRB} and Proposition \ref{pro:compatible O}, $(r_N^{\sharp})_{0}=(r_N^{\sharp})_{\lambda}|_{\lambda=0}$ is an $\huaO$-operator on the \LP pair $(A;\ad^*)$ and, furthermore, $r^\sharp_0$ and $(r_N^{\sharp})_{0}$ are compatible. By Lemma \ref{lem:equivalence r-matrix}, $r_{N}$ is  a conformal $r$-matrix. It is straightforward to check that the $\huaO$-operators $r^\sharp_0$ and $(r_N^{\sharp})_{0}$ are compatible if and only if the conformal $r$-matrices $r$ and $r_N$ are compatible. The conclusion follows.
\end{proof}

\begin{pro}\label{pro:compatible r-matrix2}
Let $(r,N)$ be a conformal $r$-matrix-Nijenhuis structure on a conformal algebra $A$ and and set $r=\sum_{i} a_i\otimes a_j$. Then for all $k\in \Nat$, $r_{N^k}=\sum_i a_i\otimes N^k (b_i)\in A\otimes A$ are pairwise compatible skew-symmetric conformal $r$-matrices.
\end{pro}
\begin{proof}
 Similar to the proof of Proposition \ref{pro:compatible r-matrix}, the conclusion follows from Theorem \ref{thm:rNijRB} and Proposition \ref{pro:compatible O2}.
\end{proof}

\subsection{Symplectic-Nijenhuis structures on Lie conformal algebras}
Recall that a $2$-form on a Lie conformal algebra $A$ is a $\Comp$-bilinear map $\omega:A\otimes A\rightarrow\Comp[\lambda]$ satisfies the following conditions:
\begin{eqnarray}
\label{eq:symplectic1}  \{\partial a_\lambda b\}_\omega &=& -\lambda \{a_\lambda b\}_{\omega},\\
\label{eq:symplectic2}   \{ a_\lambda \partial b\}_\omega &=& \lambda \{a_\lambda b\}_{\omega},\\
 \label{eq:symplectic3} \{a_\lambda b\}_\omega &=& -\{b_\lambda a\}_\omega, \forall~a,b,c\in A.
\end{eqnarray}

A  $2$-form $\omega$ on a finite Lie conformal algebra $A$ is called {\bf non-degenerate} if  $\omega^\natural$ defined by
\begin{equation}
  \{a_\lambda b\}_\omega=\langle\omega^\natural(a),b\rangle_\lambda
\end{equation}
is a $\Comp[\partial]$-module isomorphism from $A$ to $A^{\ast c}$.

Recall that a $2$-cocycle on a Lie conformal algebra $A$ with coefficients in the trivial module is a $2$-form $\omega$ on $A$ satisfying $\dM^T\omega=0$, i.e.
\begin{eqnarray}
 \label{eq:symplectic4}\{a_\lambda[b_{\mu}c]\}_\omega-\{b_\mu [a_{\lambda}c]\}_\omega &=& \{[a_{\lambda}b]_{\lambda+\mu}c\}_\omega,\quad \forall~a,b,c\in A.
\end{eqnarray}

\begin{defi}{\rm(\cite{HL15})}
  A {\bf symplectic structure} on a Lie conformal algebra $(A,[\cdot_\lambda\cdot])$  is a non-degenerate $2$-cocycle $\omega$ on $A$.
\end{defi}

\emptycomment{
\begin{ex}
  Let $\omega\in\wedge^2\g^*$ be a symplectic structure on a Lie algebra $(\g,[\cdot,\cdot])$. Then $\tilde{\omega}$ defined by
  \begin{equation}\label{eq:sym-current}
  \{f(\partial)x_\lambda g(\partial)y\}_{\tilde{\omega}}=f(-\lambda)g(\lambda)\omega(x,y),\quad x,y\in\g,f(\partial),g(\partial)\in\Comp[\partial]
  \end{equation}
  is a symplectic structure on the current Lie conformal algebra ${\rm Cur}~\g=\Comp[\partial]\otimes\g$.
\end{ex}}

\begin{lem}{\rm(\cite{HB})}\label{lem:symplectic-r-matrix}
Let $A$ be a Lie conformal algebra. Then $r\in A\otimes A$ is a skew-symmetric and non-degenerate $r$-matrix if and only if the bilinear form $\omega$ defined by
\begin{equation}
  \{a_\lambda b\}_\omega=\langle(r_{0}^{\sharp})^{-1}(a),b\rangle_{\lambda}, \quad a,b\in A
\end{equation}
is a symplectic structure on  $A$, where $r_{0}^{\sharp}\in Chom(A^{\ast c},A)$ is given by \eqref{eq:non-degenerate r-matrix}.
\end{lem}

By Lemmas \ref{lem:equivalence r-matrix} and \ref{lem:symplectic-r-matrix}, we have
\begin{cor}\label{eq:symplectic-O}
Let $A$ be a Lie conformal algebra and $\omega$ a $2$-form on $A$. Then $\omega$ is a symplectic structure on the Lie conformal algebra $A$ if and only if $(\omega^{\natural})^{-1}: A^{\ast c}\rightarrow A$ is an $\huaO$-operator on the \LP pair $(A;\ad^{\ast})$.
\end{cor}

\begin{defi}
Let $\omega$ be a symplectic structure and $N$ a Nijenhuis operator on a Lie conformal algebra $A$. Then $(\omega,N)$ is called a {\bf symplectic-Nijenhuis structure } on the Lie conformal algebra $A$ if for all $a,b\in A$,
\begin{equation}\label{eq:SN1}
\{N(a)_\lambda b\}_\omega=\{a_\lambda N(b)\}_\omega
\end{equation}
and $\omega_N:A\otimes A\rightarrow\Comp[\lambda]$ defined by
\begin{equation}
  \{a_\lambda b\}_{\omega_N}=\langle\omega^\natural(N(a)),b\rangle_\lambda
\end{equation}
is also a $2$-cocycle.
\end{defi}
Note that \eqref{eq:SN1} implies that $\omega_N$ is a $2$-form on the Lie conformal algebra $A$.

Recall that a {\bf symplectic structure} on a Lie algebra $(\g,[\cdot,\cdot])$  is a nondegenerate $2$-cocycle $\omega\in\wedge^2\g^*$, i.e.
  $$\omega([x,y],z)+\omega([y,z],x)+\omega([z,x],y)=0,\quad\forall~x,y,z\in\g.$$

Let $\omega$ be a symplectic structure and $N$ a Nijenhuis operator on a Lie algebra $\g$. Recall that a {\bf symplectic-Nijenhuis structure} on a Lie algebra $\g$ is a pair $(\omega,N)$ satisfying that for $x,y\in\g$,
$$\omega(N(x),y)=\omega(x,N(y))$$
and $\omega_N:\g\otimes \g\rightarrow\g$ defined by
$$\omega_N(x,y)=\omega(N(x),y)$$
is also a $2$-cocycle on the Lie algebra $\g$.

\begin{ex}
  Let $(\omega,N)$ be a symplectic-Nijenhuis structure on a Lie algebra $\g$. Then $(\tilde{\omega},\tilde{N})$ is a symplectic-Nijenhuis structure on the current Lie conformal algebra ${\rm Cur}~\g=\Comp[\partial]\otimes\g$, where $\tilde{\omega}$ is defined by
   \begin{equation}\label{eq:sym-current}
  \{f(\partial)x_\lambda g(\partial)y\}_{\tilde{\omega}}=f(-\lambda)g(\lambda)\omega(x,y),\quad x,y\in\g,f(\partial),g(\partial)\in\Comp[\partial],
  \end{equation}
and $\tilde{N}$ is given by \eqref{ex:Nij-current}.
\end{ex}

 \begin{ex}\label{ex:SN2}
  Given a Lie conformal algebra $A=\Comp[\partial]a\oplus \Comp[\partial]b$ with the $\lambda$-bracket given by
  $$[a_\lambda a]=(\partial+2\lambda)a,\quad [a_\lambda b]=(\partial+k\lambda+l)b,\quad [b_\lambda b]=0,\quad k,l\in\Comp.$$
  Define  $\rho:A\rightarrow{\rm gc}(A)$ by
  $$\rho(a)_\lambda(a)=(\partial+\lambda+m)a,\quad \rho(a)_\lambda(b)=(\partial+k\lambda+l)b,\quad \rho(b)_\lambda(a)=\rho(b)_\lambda(b)=0,$$
  where $m\in \Comp$. It is straightforward to check that $\rho$ gives a module of $A$ over $A$.  Consider the semi-direct Lie conformal algebra $D=A\ltimes_{\rho^*}A^{*c}$, $\omega$ defined by
  \begin{align*}
\{a_\lambda b\}_\omega=0,\quad  \{a_\lambda a^*\}_\omega=-1,\quad \{a_\lambda b^*\}_\omega=0,\quad \{b_\lambda a^*\}_\omega=0,\quad  \{b_\lambda b^*\}_\omega=-1
  \end{align*}
  is a symplectic structure on the Lie conformal algebra $D$, where $\{a^*,b^*\}$ is the dual $\Comp[\partial]$-basis of $A^{*c}$. Then $N:D\rightarrow D$ defined by
  $$N(a)=k_1a,\quad N(b)=k_2 b,\quad N(a^*)=k_1 a^*,\quad N(b^*)=k_2 b^*$$
  is a Nijenhuis operator on the Lie conformal algebra $D$. Furthermore, $(\omega,N)$ is a symplectic-Nijenhuis structure on the Lie conformal algebra $D$.
\end{ex}

\begin{ex}\label{ex:SN3}
  Given a Lie conformal algebra $A=\Comp[\partial]a\oplus \Comp[\partial]b$ with the $\lambda$-bracket given by
  $$[a_\lambda a]=0,\quad [a_\lambda b]=\lambda a,\quad [b_\lambda b]=\partial b+2\lambda b.$$
  Then $\omega$ defined by
  \begin{align*}
\{a_\lambda b\}_\omega&=\lambda^2 b
  \end{align*}
  is a symplectic structure on the Lie conformal algebra $A$. It is obvious that $N:A\rightarrow A$ defined by
  $$N(a)=k_1a,\quad N(b)=k_1 b$$
  is a Nijenhuis operator on the Lie conformal algebra $A$. Furthermore, $(\omega,N)$ is a symplectic-Nijenhuis structure on the Lie conformal algebra $A$.
\end{ex}

\begin{pro}
If $(\omega, N)$ is a symplectic-Nijenhuis structure on $A$, then for any $k\in \Nat$, $\omega_{N^k}$ is closed, i.e. $\dM^T\omega_{N^k}=0 $, where $\omega_{N^k}:A\otimes A\rightarrow\Comp[\lambda]$ is defined by
\begin{equation}
  \{a_\lambda b\}_{\omega_{N^k}}=\langle\omega^\natural(N^k(a)),b\rangle_\lambda.
\end{equation}
\end{pro}
\begin{proof}
  By a direct calculation, for $a,b,c\in A$, we have
  \begin{eqnarray*}
    \{a_{\lambda_1}b_{\lambda_2}c\}_{\dM ^T\omega_{N^{k+1}}}&=&\{N(a)_{\lambda_1}b_{\lambda_2}c\}_{\omega_{N^{k}}}+\{a_{\lambda_1}N(b)_{\lambda_2}c\}_{\omega_{N^{k}}}-\{N(a)_{\lambda_1}N(b)_{\lambda_2}c\}_{\omega_{N^{k-1}}}\\
    &&+\{\big(N\big([N(a)_{\lambda_1}(b)]+[a_{\lambda_1}N(b)]-N[a_{\lambda_1}b]\big)-[N(a)_{\lambda_1}N(b)]\big)_{\lambda_1+\lambda_2}N^{k-1}(c)\}_\omega.
  \end{eqnarray*}
  Since $N$ is a Nijenhuis operator on the Lie conformal algebra $A$, the above equality implies
  $$\{a_{\lambda_1}b_{\lambda_2}c\}_{\dM ^T\omega_{N^{k+1}}}=\{N(a)_{\lambda_1}b_{\lambda_2}c\}_{\omega_{N^{k}}}+\{a_{\lambda_1}N(b)_{\lambda_2}c\}_{\omega_{N^{k}}}-\{N(a)_{\lambda_1}N(b)_{\lambda_2}c\}_{\omega_{N^{k-1}}}.$$
  By induction on $k$, the conclusion follows.
\end{proof}

There is a close relationship between symplectic-Nijenhuis structures and $\GRBN$-structures.

\begin{thm}\label{thm:SN-ON}
Let $\omega:A\otimes A\rightarrow\Comp[\lambda]$ be a non-degenerate $2$-form and $N:A\rightarrow A$ a Nijenhuis operator on a finite Lie conformal algebra $A$. Then $(\omega,N)$ is a symplectic-Nijenhuis structure on $A$ if and only if $((\omega^{\natural})^{-1},N,S=N^{\ast})$ is an $\GRBN$-structure on the \LP pair $(A;\ad^{\ast})$.
\end{thm}
\begin{proof}
Assume that $(\omega,N)$ is a symplectic-Nijenhuis structure on $A$. By \eqref{eq:symplectic3}, we have
 $$\langle\omega^\natural(a),b\rangle_\lambda=-\langle\omega^\natural(b),a\rangle_\lambda.$$
 Since $\omega$ is non-degenerate, letting $a=(\omega^\natural)^{-1}(\alpha)$ and $b=(\omega^\natural)^{-1}(\beta)$, we also have
 \begin{equation}\label{eq:conformal pair}
  \langle\alpha,(\omega^\natural)^{-1}(\beta)\rangle_\lambda=-\langle\beta,(\omega^\natural)^{-1}(\alpha)\rangle_\lambda.
 \end{equation}
By \eqref{eq:SN1}, letting $a=(\omega^\natural)^{-1}(\alpha)$ and $b=(\omega^\natural)^{-1}(\beta)$, we have
$$\langle\omega^\natural(N(\omega^\natural)^{-1}(\alpha)),(\omega^\natural)^{-1}(\beta)\rangle_\lambda=-\langle\alpha,N((\omega^\natural)^{-1}(\beta))\rangle_\lambda.$$
Furthermore, by \eqref{eq:conformal pair}, we have
$$\langle\alpha,(\omega^\natural)^{-1}(N^*\beta)\rangle_\lambda=\langle\alpha,N((\omega^\natural)^{-1}(\beta))\rangle_\lambda,$$
 which implies that $(\omega^\natural)^{-1}\circ N^*=N\circ (\omega^\natural)^{-1}$ holds.

By Corollary \ref{eq:symplectic-O}, $\omega$ is a symplectic structure on the Lie conformal algebra $A$ if and only if $(\omega^{\natural})^{-1}: A^{\ast c}\rightarrow A$ is an $\huaO$-operator on the \LP pair $(A;\ad^{\ast})$. In the following, we need to show that
\begin{equation}\label{eq:ttt}
[\alpha_\lambda\beta]^{N\circ (\omega^\natural)^{-1}}= [\alpha_\lambda\beta]^{(\omega^\natural)^{-1}}_{N^*},\quad \forall~ \alpha,\beta\in A^{* c}.
\end{equation}
 By the fact that $\omega$ is a $2$-cocycle,  letting $a=(\omega^\natural)^{-1}(\alpha)$ and $b=(\omega^\natural)^{-1}(\beta)$,  we have
\begin{eqnarray*}
  &&\langle[\alpha_\lambda\beta]^{N\circ (\omega^\natural)^{-1}}- [\alpha_\lambda\beta]^{(\omega^\natural)^{-1}}_{N^*},c\rangle_\mu\\
  &=&\langle \ad^*((\omega^\natural)^{-1}(\alpha))_\lambda N^*(\beta)-\ad^*((\omega^\natural)^{-1}(\beta))_{-\lambda-\partial} N^*(\alpha)-N^*(\ad^*((\omega^\natural)^{-1}(\alpha))_\lambda \beta)\\
  &&+N^*(\ad^*((\omega^\natural)^{-1}(\beta))_{-\lambda-\partial} \alpha), c   \rangle_\mu\\
  &=&\langle\ad^*(a)_\lambda N^*\omega^\natural(b)-\ad^*(b)_{-\lambda-\partial}N^*\omega^\natural(a)-N^*(\ad^*(a)_\lambda\omega^\natural(b))+N^*(\ad^*(b)_{-\lambda-\partial}\omega^\natural(a)),c\rangle_\mu\\
  &=&\langle\ad^*(a)_\lambda N^*\omega^\natural(b)-\ad^*(b)_{\mu-\lambda}N^*\omega^\natural(a)-N^*(\ad^*(a)_\lambda\omega^\natural(b))+N^*(\ad^*(b)_{\mu-\lambda}\omega^\natural(a)),c\rangle_\mu\\
  &=&-\{b_{\mu-\lambda}N[a_\lambda c]\}_\omega+\{a_\lambda N[b_{\mu-\lambda}c]\}_\omega+\{b_{\mu-\lambda}[a_\lambda N (c)]\}_\omega-\{a_\lambda [b_{\mu-\lambda}N(c)]\}_\omega\\
  &=&-\{b_{\mu-\lambda}N[a_\lambda c]\}_\omega+\{a_\lambda N[b_{\mu-\lambda}c]\}_\omega-\{[a_\lambda b]_\mu N(c)\}_\omega\\
  &=&-\{N(b)_{\mu-\lambda}[a_\lambda c]\}_\omega+\{N(a)_\lambda [b_{\mu-\lambda}c]\}_\omega-\{N[a_\lambda b]_\mu c\}_\omega\\
  &=&-\{b_{\mu-\lambda}[a_\lambda c]\}_{\omega_N}+\{a_\lambda [b_{\mu-\lambda}c]\}_{\omega_N}-\{[a_\lambda b]_\mu c\}_{\omega_N}.
\end{eqnarray*}
Since $\omega_N$ is a $2$-cocycle, \eqref{eq:ttt} follows. Thus $((\omega^{\natural})^{-1},N,S=N^{\ast})$ is an $\GRBN$-structure on the \LP pair $(A;\ad^{\ast})$.

The converse can be proved similarly. We omit the details.
\end{proof}

By Proposition \ref{pro:non-degenerate} and Theorem \ref{thm:SN-ON}, we have
\begin{cor}\label{cor:SN-rN}
Let $\omega$ be a nondegenerate $2$-form and $N:A\rightarrow A$ a Nijenhuis operator on a finite Lie conformal algebra $A$. Then $(\omega, N)$ is symplectic-Nijenhuis structure on $A$ if and only if $(r,N)$ is a conformal $r$-matrix-Nijenhuis structure on $A$, where $r\in A\otimes A$ is defined by
\begin{equation}\label{eq:r-matrix-symplectic}
\langle\beta,(\omega^{\natural})^{-1}(\alpha)\rangle_{\lambda}=\langle\alpha\otimes \beta,r\rangle_{(-\lambda,\lambda)} \quad\forall~ \alpha,\beta\in A^{\ast c}.
\end{equation}
\end{cor}

By Proposition \ref{pro:compatible r-matrix} and Corollary \ref{cor:SN-rN}, we have
\begin{cor}
If $(\omega, N)$ is a symplectic-Nijenhuis structure on a Lie conformal algebra $A$, then $r$ and $r_N$  are compatible $r$-matrices, where $r$ is given by \eqref{eq:r-matrix-symplectic} and $r_N$ is given by
\begin{equation}
\langle\beta,N(\omega^{\natural})^{-1}(\alpha)\rangle_{\lambda}=\langle\alpha\otimes \beta,r_N\rangle_{(-\lambda,\lambda)} \quad\forall~ \alpha,\beta\in A^{\ast c}.
\end{equation}
\end{cor}

By Proposition \ref{pro:compatible r-matrix2} and Corollary \ref{cor:SN-rN}, we have
\begin{cor}\label{cor:CNcom}
If $(\omega, N)$ is a symplectic-Nijenhuis structure on a Lie conformal algebra $A$, then for any $k,l\in\Nat$, $r_{N^k}$ and $r_{N^l}$ defined by
  \begin{eqnarray*}
\langle\beta,N(\omega^{\natural})^{-1}(\alpha)\rangle_{\lambda}=\langle\alpha\otimes \beta,r_N\rangle_{(-\lambda,\lambda)},\quad
 \langle\beta,N(\omega^{\natural})^{-1}(\alpha)\rangle_{\lambda}=\langle\alpha\otimes \beta,r_N\rangle_{(-\lambda,\lambda)},\quad\forall~ \alpha,\beta\in A^{\ast c}
\end{eqnarray*}
  are pairwise compatible $r$-matrices.
\end{cor}

\emptycomment{\begin{cor}
If $(\omega, N)$ is symplectic-Nijenhuis structure on a Lie conformal algebra $A$, then $\omega_{N^k}$ are closed, i.e. $\dM^T\omega_{N^k}=0 $, where $\omega_{N^k}:A\otimes A\rightarrow\Comp[\lambda]$ defined by
\begin{equation}
  \{a_\lambda b\}_{\omega_N^k}=\langle\omega^\natural(N^k(a)),b\rangle_\lambda.
\end{equation}
\end{cor}}

By Theorems \ref{thm:ON construction} and \ref{thm:SN-ON}, we have
\begin{cor}
Let $r$ be a conformal $r$-matrix and $\omega$ a symplectic structure on a Lie conformal algebra $A$. Define $\tilde{r}\in A\otimes A$ by
\begin{equation*}
\langle\beta,(\omega^{\natural})^{-1}(\alpha)\rangle_{\lambda}=\langle\alpha\otimes \beta,\tilde{r}\rangle_{(-\lambda,\lambda)} \quad\forall~ \alpha,\beta\in A^{\ast c}.
\end{equation*}
If $r$ and $\tilde{r}$ are compatible $r$-matrices, then $(\omega,N=r_0^{\sharp}\circ\omega^{\natural})$ is a symplectic-Nijenhuis structure on the Lie conformal algebra $A$.
\end{cor}

By Proposition \ref{pro:GRBN-invertible} and Theorem \ref{thm:SN-ON}, we have
\begin{cor}
If $(\omega, N)$ is a symplectic-Nijenhuis structure on a Lie conformal algebra $A$, then  for any $k\in\Nat$, $(\omega,N^k)$ is a symplectic-Nijenhuis structure on the Lie conformal algebra $A$.
\end{cor}

\noindent
{\bf Acknowledgements. } This research was supported by the National Key Research and Development Program of China (2021YFA1002000), the National Natural Science Foundation of China (11901501), the China Postdoctoral Science Foundation (2021M700750) and the Fundamental Research Funds for the Central Universities (2412022QD033).

\end{document}